\documentclass[12pt,reqno]{amsart}
\usepackage{amsthm,amsfonts,amssymb,euscript}

\theoremstyle{plain}
  \newtheorem{theorem}[subsection]{Theorem}
  
  \newtheorem{proposition}[subsection]{Proposition}
  \newtheorem{lemma}[subsection]{Lemma}
  \newtheorem{corollary}[subsection]{Corollary}

\def\bphi{{\boldsymbol{\phi} }}
\def\bpsi{{\boldsymbol{\psi} }}
\def\F{{\mathcal F}}
\def\R{{\mathbb{R}}}
\def\C{{\mathbb{C}}}
\def\I{{\mathcal I}}

\def\A{{\mathcal A}}
\def\B{{\mathcal B}}
\def\P{{\mathcal P}}
\def\Q{{\mathcal Q}}
\def\H{{\mathcal H}}
\def\N{{\mathcal N}}

\def\M{{\mathcal M}}
\def\E{{\mathcal E}}
\def\V{{\mathcal V}}

\def\ch{\mbox{ch} (k)}
\def\chbb{\mbox{chb} (k)}
\def\trigh{\mbox{trigh} (k)}

\def\sh{\mbox{sh} (k)}
\def\shh{\mbox{sh}}
\def\chh{\mbox{ch}}
\def\shhb{\overline {\mbox{sh}}}
\def\th{\mbox{th} (k)}
\def\chhb{\overline {\mbox{  ch}}}
\def\shbb{\mbox{shb} (k)}
\def\chb{\overline{\mbox{\rm ch}  (k)}}
\def\shb{\overline{\mbox{sh}(k)}}
\def\z{\zeta}
\def\zb{\overline{\zeta}}
\def\FF{{\mathbb{F}}}
\def\L{\Lambda}

\def\phib{\overline{\phi}}
\theoremstyle{remark}
  \newtheorem{remark}[subsection]{Remark}
  
\theoremstyle{definition}
  \newtheorem{definition}[subsection]{Definition}

\begin{document}

\def\vac{{\big\vert 0\big>}}
\def\fpsi{{\big\vert\psi\big>}}
\def\bphi{{\boldsymbol{\phi} }}
\def\bpsi{{\boldsymbol{\psi} }}
\def\F{{\mathcal F}}
\def\R{{\mathbb{R}}}
\def\C{{\mathbb{C}}}
\def\I{{\mathcal I}}
\def\Q{{\mathcal Q}}
\def\ch{\mbox{\rm ch} (k)}
\def\cht{\mbox{\rm ch} (2k)}
\def\chbb{\mbox{\rm chb} (k)}
\def\trigh{\mbox{\rm trigh} (k)}
\def\p{\mbox{p} (k)}
\def\sh{\mbox{\rm sh} (k)}
\def\sht{\mbox{\rm sh} (2k)}
\def\shh{\mbox{\rm sh}}
\def\chh{\mbox{ \rm ch}}
\def\th{\mbox{\rm th} (k)}
\def\thb{\overline{\mbox{\rm th} (k)}}
\def\shhb{\overline {\mbox{\rm sh}}}
\def\chhb{\overline {\mbox{\rm ch}}}
\def\shbb{\mbox{\rm shb} (k)}
\def\chbt{\overline{\mbox{\rm ch}  (2k)}}
\def\shb{\overline{\mbox{\rm sh}(k)}}
\def\shbt{\overline{\mbox{\rm sh}(2k)}}
\def\z{\zeta}
\def\zb{\overline{\zeta}}
\def\FF{{\mathbb{F}}}
\def\S{{\bf S}}
\def\Sr{{\bf S}_{red}}
\def\W{{\bf W}}
\def\WW{{\mathcal W}}
\def\N{{\bf N}}
\def\L{{\mathcal L}}
\def\E{{\mathcal E}}
\def\X{{\tilde X}}
\def\Nor{{\rm Nor}}

\title[Pair excitations, II]%
{ Pair excitations and the mean field approximation of interacting Bosons, II}

\author{M. Grillakis}
\address{University of Maryland, College Park}
\email{mng@math.umd.edu}

\author{M. Machedon}
\address{University of Maryland, College Park}
\email{mxm@math.umd.edu}

\subjclass{}
\keywords{}
\date{}
\dedicatory{}
\commby{}

\maketitle

\begin{abstract}
We consider a large number of Bosons with interaction potential $v_N(x)=N^{3 \beta}v(N^{\beta}x)$.
In our earlier papers \cite{GMM1}-\cite{GM} we considered a set of equations for the condensate
$\phi$ and pair excitation function $k$
and proved that they provide a Fock space approximation to the exact evolution of a coherent state for $\beta <\frac{1}{3}$.
In \cite{GMM}, in the hope of treating higher values of $\beta <1$, we introduced a coupled refinement of our original equations. In that paper, we showed the coupled equations conserve the number of particles and energy.
In the current paper we prove that the coupled equations do indeed provide a Fock space approximation  for $\beta < \frac{2}{3}$, at least locally in time.
In order to do that, we re-formulate the equations of \cite{GMM} in a way reminiscent of BBGKY and apply harmonic analysis techniques in the spirit of those used by X. Chen and J. Holmer in \cite{C-H1}
to prove the necessary estimates. In turn, these estimates provide bounds for the pair excitation function $k$.
While our earlier papers provide background material, the  methods of this paper paper are mostly new, and the presentation is self-contained.
\end{abstract}

\section{Introduction}

The problem considered in this paper (as well as our earlier papers \cite{GMM1}- \cite{GMM}) is  the $N$-body linear Schr\"odinger equation
\begin{align}
&\left(\frac{1}{i} \frac{\partial}{\partial t}
-\sum_{j=1}^N \Delta_{x_j}+\frac{1}{N}\sum_{i<j}v_N(x_i-x_j)\right)\psi_N(t, \cdot)=0
  \label{schl}\\
&\psi_N (0, x_1, \cdots, x_N) \sim \phi_0( x_1) \phi_0( x_2) \cdots \label{sim}
\phi_0( x_N)\\
&  \|\psi_N(t, \cdot)\|_{ L^2(\mathbb R^{3N})}=1 \notag
\end{align}
where $v_{N}(x):=N^{3\beta}v(N^{\beta}x)$ with $0\leq\beta\leq 1$,
$v\in \mathcal S$ and $v\geq 0$. The meaning of $\sim$ in \eqref{sim} will  be made precise.

The  goal is to find a rigorous, simple  approximation (in a suitable norm) to $\psi_N$ which is consistent with
\begin{align}
\psi_{approx} (t, x_1, \cdots, x_N) \sim e^{i \chi(t)}\phi(t, x_1) \phi(t, x_2) \label{manyN} \cdots
\phi(t, x_N)
\end{align}
as $N \to \infty$,
where
$\phi$ (which represents the Bose-Einstein condensate)  satisfies a  non-linear Schr\"odinger equation.
The problem becomes  more difficult, interesting, and requires new ideas as $\beta$ approaches 1,
and this explains why several authors, (including us) devoted several papers to  their programs.

We refer to \cite{lieb05} for extensive background on (static) Bose-Einstein condensation.

During recent years, in a series of papers by
 Erd\"os and Yau \cite{E-Y1}, and
 Erd\"os, Schlein and Yau \cite{E-S-Y1} to \cite{E-S-Y4}, it was proved
 \begin{align}
\gamma_1^N(t, x, x') \to \overline \phi(t, x){\phi}(t, x') \label{ESY}
\end{align}
in trace norm as $N \to \infty$, and similarly for the higher order marginal density matrices $\gamma_k^N$, where $k$ is fixed.
 Recent simplifications and generalizations
 were given in \cite{K-MMM}, \cite{KSS}, \cite{CP}, \cite{XC3},     \cite{C-H1}, \cite{C-H2}. See also
 \cite{FKS}, \cite{K-P} for a different approach.

 We also mention the approach based on the quantum de Finetti theorem in
 \cite{CHNS}, as well as results in the case of negative interaction potentials
 in \cite{XC-H-3} and the different approach of  Knowles and Pickl \cite{K-P}.

Another approach to this problem is based on Fock space techniques and the second quantization. In physics, it was pioneered
in the  papers by Bogoliubov \cite{Bog47},
 Lee, Huang and Yang \cite{Lee-H-Y} in the static case, and Wu \cite{wuI} in the time dependent case. See also the more recent papers  \cite{MargetisI}, \cite{margetisII}.

 In the rigorous mathematical literature devoted to the time evolution problem, it originates in the of work  Hepp \cite{hepp},  Ginibre and Velo \cite{G-V}
and, after lying dormant for about 30 years,   Rodnianski and Schlein  \cite{Rod-S}, followed by
\cite{GMM1}.
 Currently, it is an active field.

Our project, initiated in collaboration with Margetis in \cite{GMM1}, is to study a PDE describing additional second order corrections (given by a Bogoliubov transformation  $e^{\B}$) to the right hand side of the approximation \eqref{manyN}.
Mathematically, Boguliubov transformations are representations of a group isomorphic to a real symplectic group, corresponding to the  Segal-Shale-Weil  representation in infinite dimensions, due to Shale. Interestingly, the theories seem to have evolved independently
in physics and pure mathematics.

 Several important recent papers also use coherent states and Bogoliubov transformations. These include \cite{B-O-S} and  \cite{B-K-S}. In fact, Theorem 2.2 in \cite{B-K-S} proves that the unitary operators of the type used in \cite{hepp} and \cite{G-V} can be obtained, abstractly, as  Bogoliubov transformations. We mention in passing that our concrete\footnote{We mean that $\B =\B(k)$ where $k$ satisfies a PDE in 6+1 dimensions and is, in principle, computable numerically .} in $e^{\B}$ (to be discussed below)  agrees with that Bogoliubov transformation (up to a phase), but only when applied to the vacuum.
 Roughly speaking, our $e^{\B}$ is not the operator corresponding to the evolution of the quadratic Bogoliubov Hamiltonian, but rather diagonalizes it.

In the current paper we initiate the analysis of solutions to a coupled system of PDEs for the condensate $\phi$ and pair
excitation $k$, see\eqref{pphieq},
\eqref{sseq}-\eqref{wweq} below.
It would be very interesting to obtain estimates  up to the case $\beta <1 $, as well as large time estimates for the solutions of these equations which are uniform in $N$, similar to those
obtained in earlier works for the uncoupled equations (\eqref{meanfield-3}-\eqref{pair-W-1} below) which describe the case $\beta< 1/3$. We hope to address these question in future work. In this paper we only consider the case $1/3<\beta < 2/3$, locally in time. As pointed out by one of our referees, it is unlikely that this type of construction, based on coherent states and Bogoliubov transformations, will work in the case $\beta = 1$. This is because the construction is not sufficient to capture the ground state energy for the many body (static) case. The references for this phenomenon are \cite{ESYstatic}, \cite{Y-Y}.

See also the paper by Lewin,   Nam, and  Schlein
 \cite{LPS} for a Fock space type approach to an $L^2(\mathbb R^{3N})$ estimate based on corrections to a pure
 tensor product (Hartree state) rather than a coherent state. That approach has been generalized very recently to the case $\beta < 1/3$
 by Nam and Napiorkowski
 in \cite{N-N},
 and their work uses the linear equations \eqref{pair-S-1}, \eqref{pair-W-1} introduced in \cite{GM}, as well as some of the estimates from that paper.

Very recently, after completing this work, we have also learned of the related  paper \cite{BCS}. A brief comparison of the results of that paper with ours is included at the end next section.

 Also very recently we  learned  that,
independently and in a different framework, Bach, Breteaux, Chen, Fr\"ohlich and Sigal derived equations closely related to
the equations of \cite{GMM} and section \eqref{self-contained} of our current  paper
in their recent work \cite{BBCFS}. Those equations become equivalent to ours
in the case of pure states.

\section{Background and statement of the main result}

We start with a very brief review of symmetric Fock space. This is included for the convenience of the reader, and follows closely
the exposition from our earlier papers.
We refer the reader to \cite{GMM} for more details and comments.   The elements of
$\F$ are  vectors of the form
\begin{equation*}
\psi=\big(\psi_{0}\ ,\ \psi_{1}(x_{1})\ ,\ \psi_{2}(x_{1},x_{2})\ ,\ \ldots\ \big)
\end{equation*}
where $\psi_{0}\in \C$ and $\psi_k$ are symmetric $L^2$ functions. The inner product is
\begin{equation*}
\big<\phi, \psi\big> =\overline \phi_0 \psi_{0}+
\sum_{n=1}^{\infty}\int \overline \phi_n \psi_n\ .
\end{equation*}
Thus we use physicists' convention of an inner product linear in the {\it second} variable.
The creation and annihilation distribution valued operators denoted by $a^{\ast}_{x}$ and
$a_{x}$ respectively which act on vectors of the form $(0, \cdots, \psi_{n-1}, 0,  \cdots)$ and
$(0, \cdots, \psi_{n+1}, 0,  \cdots)$
by
\begin{align*}
&a^{\ast}_{x}(\psi_{n-1}):=\frac{1}{\sqrt{n}}\sum_{j=1}^{n}\delta(x-x_{j})
\psi_{n-1}(x_{1},\ldots ,x_{j-1},x_{j+1},\ldots ,x_{n})
\\
&a_{x}(\psi_{n+1}):=\sqrt{n+1}\psi_{n+1}([x],x_{1},\ldots ,x_{n})
\end{align*}
with $[x]$ indicating that the variable $x$ is frozen.
The vacuum state is defined as follows:
\begin{equation*}
\Omega :=(1,0,0\ldots )
\end{equation*}
 One can easily check the canonical relations $\big[a_{x},a^{\ast}_{y}\big]=\delta(x-y)$ and since
the creation and annihilation operators are distribution valued we can form operators that act on $\F$ by introducing
a field, say $\phi(x)$, and form
\begin{align*}
a (\bar{\phi}):=\int dx\left\{\bar{\phi}(x)a_{x}\right\}
\quad {\rm and}\quad
a^{\ast} (\phi):=\int dx\left\{\phi(x)a^{\ast}_{x}\right\}
\end{align*}
where by convention we associate $a$ with $\bar{\phi}$ and  $a^{\ast}$ with $\phi$.
Also define the skew-Hermitian operator
\begin{equation}
\A(\phi):=\int dx\left\{\bar{\phi}(x)a_{x}-\phi(x)a^{\ast}_{x}\right\} \label{meanfield-2}
\end{equation}
and coherent states
\begin{equation}\label{cohstates}
\psi:=e^{-\sqrt{N}\A(\phi)}\Omega \ .
\end{equation}
It is easy to check that
\begin{equation*}
e^{-\sqrt{N}\A(\phi)}\Omega =\left(\ldots\ c_{n}\prod_{j=1}^{n}\phi(x_{j})\ \ldots\right)
\quad {\rm with}\quad  c_{n}=\big(e^{-N \|\phi\|^2_{L^2}}N^{n}/n!\big)^{1/2}\ .
\end{equation*}

We also consider
\begin{align}
&\B(k):=
\frac{1}{2}\int dxdy\left\{\bar{k}(t,x,y)a_{x}a_{y}
-k(t,x,y)a^{\ast}_{x}a^{\ast}_{y}\right\}\ . \label{pairexcit-2}
\end{align}

This particular construction and the corresponding unitary operator $\M:=e^{-\sqrt{N}\A}e^{-\B}$  were introduced
(at least in the mathematics literature related to the problem under consideration)
in \cite{GMM1}.
The construction is in the spirit of Bogoliubov theory in physics, and  the Segal-Shale-Weil representation in mathematics.

The Fock Hamiltonian  is
\begin{subequations}
\begin{align}
&\H:=\H_{1}-\frac{1}{N}\V\quad\quad\quad   {\rm where,} \label{FockHamilt1-a}\\
&\H_{1}:=\int dxdy\left\{
\Delta_{x}\delta(x-y)a^{\ast}_{x}a_{y}\right\}\quad {\rm and} \label{FockHamilt1-b}\\
&\V:=\frac{1}{2}\int dxdy\left\{v_{N}(x-y)a^{\ast}_{x}a^{\ast}_{y}a_{x}a_{x}\right\}\ ,
\label{FockHamilt1-c}
\end{align}
\end{subequations}
where $v_N (x)=N^{3\beta}v\big(N^{\beta}x\big)$.
It is a diagonal operator on Fock space, and it acts as a differential operator in $n$ variable
\begin{align*}
H_{n, \, \, PDE}=\sum_{j=1}^{n}\Delta_{x_{j}} - \frac{1}{N}
\sum_{i<j}N^{3\beta}v\big(N^{\beta}(x_{j}-x_{k})\big)
\end{align*}
 on the $n$th component of $\F$. Notice this is the same as \eqref{schl}, except that the dimension $n$ is decoupled from
 the parameter $N$.

Our goal is to study the evolution of (possibly modified) coherent initial conditions of the form
\begin{align}
 \psi_{exact}=e^{i t \H}e^{-\sqrt{N}\A(\phi_{0})}e^{-\B(k_{0})}\Omega \label{exact}
 \end{align}
 In our earlier papers\cite{GMM1, GMM2, GM, GMM} we considered an approximation of the form
\begin{equation}\label{approx}
\psi_{appr}:=e^{-\sqrt{N}\A(\phi(t)}e^{-\B(k(t))}\Omega
\end{equation}
and derived suitable Schr\"odinger type equations equations for $\phi(t, x)$, $k(t, x, y)$
so that
 $\psi_{exact}(t)\approx e^{iN\chi(t)}\psi_{appr}(t)$,
 with $\chi(t)$ a real phase factor, in order to find
  precise estimates in Fock space, see Theorem \eqref{GMthm} below.
  Our strategy is to consider
  \begin{align*}
  \psi_{red}=e^{\B(t)}e^{\sqrt{N}\A(t)}e^{it\H}e^{-\sqrt{N}\A(0)}e^{-\B(0)}\Omega
  \end{align*}
 and then  compute a "reduced Hamiltonian"
 \begin{align}
 H_{red} = \frac{1}{i}\big(\partial_{t}\M^{\ast}\big)\M +\M^{\ast}\H\M \label{reduced}
 \end{align}
 so that
\begin{equation}
\frac{1}{i}\partial_{t}\psi_{red}=
\H_{red}\psi_{red} \ .
\end{equation}
To state the results of \cite{GMM1, GMM2, GM} we define the operator kernel
\begin{align}
&g_N(t, x, y):= -\Delta_x \delta (x-y)
 +(v_N * |\phi|^2 )(t, x) \delta(x-y) \nonumber \\
&\qquad\qquad\  +v_N(x-y) \overline\phi (t, x)  \phi(t, y)\label{op-g}
\end{align}
and
\begin{align}
&{\bf S}_{old}(s):= \frac{1}{i} s_t  + g_N^T \circ s + s \circ g_N\quad {\rm and} \notag\\
&{\bf W}_{old}(p):= \frac{1}{i} p_t  +[g_N^T, p]\label{wold}\\
&m_N(x, y) :=- v_N(x-y) \phi(x) \phi(y) \, \, ,
v_N(x)= N^{3 \beta} v(N^{\beta}x)\notag
\end{align}
The main result of \cite{GM}  can be summarized as follows:
\begin{theorem}\label{GMthm} Let $\phi$ and $k$ satisfy
\begin{subequations}
\begin{align}
&\frac{1}{i}\partial_{t}\phi -\Delta\phi +\big(v_{N}\ast\vert\phi\vert^{2}\big)\phi =0 \label{meanfield-3}
\\
&   {\bf S}_{old}\left(\sht\right)=m_{N} \circ \cht + \chbt \circ m_{N}  \label{pair-S-1} \\
&   {\bf W}_{old}\left(\chbt\right)= m_{N} \circ \shbt- \sht \circ \overline m_{N}\ . \label{pair-W-1}
\end{align}
\end{subequations}
with prescribed initial conditions $\phi(0, \cdot)= \phi_0$, $k(0, \cdot, \cdot)=0$. If $\phi$, $k$ satisfy the above equations, then there exists a real phase function $\chi$ such that
\begin{equation}\label{mainestim}
\big\Vert\psi_{exact}(t)-
e^{iN\chi(t)}\psi_{appr}(t)\big\Vert_{\F}\leq \frac{C (1+t) \log^4(1+t)}{N^{(1-3\beta)/2}}\ .
\end{equation}
provided $0 < \beta < \frac{1}{3}$.
\end{theorem}
See \cite{GM} for the reasons behind these equations. We also mention a very recent simple derivation of these equations in
\cite{N-N}.
This result was extended to the case $\beta<\frac{1}{2}$ in \cite{elif2}, where it was also argued informally that the equations of \cite{GM} do not provide an approximation for $\beta>\frac{1}{2}$.

In the hope of obtaining an approximation for higher $  \beta $, in \cite{GMM} we introduced a coupled refinement of the system \eqref{meanfield-3}, \eqref{pair-S-1}, \eqref{pair-W-1}.

The coupled equations of  \cite{GMM} were introduced
the following way:
Since $\H_{red}$ is a fourth order polynomial in $a$ and $a^*$,
\begin{align}
\H_{red} \Omega =(X_0, X_1, X_2, X_3, X_4, 0, \cdots). \label{Xidef}
\end{align}

 The new, coupled equations for $\phi$ and $k$ that we introduce in \cite{GMM} can be written abstractly as
 \begin{align}
  X_1=0 \, \, \mbox{ and} \, \, X_2=0.\label{abstract}
  \end{align}
 It was shown there that they are Euler-Lagrange equations for the Lagrangian density  $X_0$, and that their solutions preserve
 the number of particles and the energy. See also \cite{BBCFS}.

\begin{remark} The static terms of $X_0(t)$ (not involving time derivatives) also appear in the recent paper
\cite{B-O-S}, but do not serve as a Lagrangian there.
\end{remark}

To write down the equations \eqref{abstract} explicitly in terms of $\phi$ and $k$,
we introduce
\begin{definition} Define
\begin{align}
&\Lambda(t, x_1, x_2)=\frac{1}{2N}\sht(t, x_1, x_2) +\phi(t, x_1)\phi(t,  x_2) \label{Lambda2}
\\
&\Gamma(t, x_1, x_2)=\frac{1}{N}\left(\shb \circ \sh\right)(t, x_1, x_2)+\overline{\phi}(t, x_1)\phi(t, x_2)
\label{Gamma2}
\end{align}
\end{definition}
 and the new operator kernel
\begin{align}
&\tilde g_{N}(t, x, y):= -\Delta_x \delta (x-y)
+v_N(x-y) (tr \Gamma)(t, x) \delta(x-y) \notag\\
 &+v_N(x-y) \Gamma(t, x, y) \label{gn}
\end{align}
 where $tr$ denotes trace density, and define
\begin{align}
&\tilde{\bf S}(s):= \frac{1}{i} s_t  + \tilde g_{N}^T \circ s + s \circ \tilde g_{N}\quad {\rm and}\quad \label{deftilde}
\tilde{\bf W}(p):= \frac{1}{i} p_t  +[\tilde g_{N}^T, p]
\end{align}

In this notation, the following is proved in \cite{GMM}
\begin{theorem} \label{mainthm}
The equation $X_1=0$ is equivalent to
\begin{align}
&\frac{1}{i}\partial_{t}\phi(t, x) -\Delta\phi +\int v_N(x-y)\Lambda(t, x, y) \overline \phi(t, y) dy \notag\\
&+\frac{1}{N} (v_N * Tr (\sh \circ \shb)  )(t, x)\phi(t, x) \label{pphieq}\\
&+ \frac{1}{N} \int  v_N(x-y) (\sh \circ \shb) (t, x, y)\phi(t, y) dy=0 \notag
\end{align}
Here $Tr (\sh \circ \shb)  (t, x)= (\sh \circ \shb)(t, x, x)$ denotes the trace density.

The equation $X_2=0$ is equivalent to
\begin{subequations}
\begin{align}
\tilde{\S}\left(\sht\right)+ (v_N \Lambda) \circ \cht + \chbt \circ (v_N \Lambda) =0 \label{sseq}\\
\tilde{\W}\left(\chbt\right)+(v_N \Lambda) \circ \shbt- \sht \circ (\overline {v_N \Lambda})=0 \label{wweq}
\end{align}
\end{subequations}
\end{theorem}
Notice the similarity with
\eqref{meanfield-3}, \eqref{pair-S-1}, \eqref{pair-W-1}.

Since it is difficult to prove estimates for these equations directly, we will write them down in a different, equivalent form.
The derivation will be self-contained, and in fact most of the rest of this paper is independent of our previous
work \cite{GMM1}-\cite{GMM}.

We now state the main result of our current paper.
\begin{theorem} \label{GM3thm}
Let $ \frac{1}{3}< \beta < \frac{2}{3}$ , and let the interaction potential $v \in \mathcal S$ satisfy $v \ge 0$ and $|\hat v| \le \hat w$
for some  $w \in \mathcal S$. Let $\phi$, $k$ be solutions to \eqref{pphieq}, \eqref{sseq}, \eqref{wweq}
with smooth initial conditions $\phi(0, \cdot) $, $k(0, \cdot) $ satisfying the following regularity
uniformly in $N$
(expressed in terms of  $\phi$, $\Lambda$ and $\Gamma$ defined above, as well as $\sh $): For some $\epsilon_0 >0$ and all
 $0 \le i \le 1$, $0 \le j \le 2$
\begin{align}
&\|<\nabla_x>^{\frac{1}{2}+\epsilon_0} \left(\frac{\partial}{\partial t}\right)^i\nabla_{x} ^j\phi(t, x)\bigg|_{t=0}\|_{L^2(dx)} \le C \label{IVP1}\\
&\|<\nabla_x>^{\frac{1}{2}+\epsilon_0} <\nabla_y>^{\frac{1}{2}+\epsilon_0}\left(\frac{\partial}{\partial t}\right)^i\nabla_{x+y}^j \Gamma(t, x, y)\bigg|_{t=0}\|_{L^2(dx dy)} \le C\label{IVP2} \\
&\|<\nabla_x>^{\frac{1}{2}+\epsilon_0} <\nabla_y>^{\frac{1}{2}+\epsilon_0} \left(\frac{\partial}{\partial t}\right)^i\nabla_{x+y} ^j\Lambda(t, x, y)\bigg|_{t=0}\|_{L^2(dx dy)} \le C \label{IVP3}\\
&\|   \nabla_{x+y} ^j   \sh(0, x, y)\|_{L^2(dxdy)} \le C \label{IVP4}
\end{align}

Then there exists a real function $\chi(t)=\chi_N(t)$ and  a (small) $T_0 >0$ and $C=C(T_0, \epsilon_0, \beta)$  such that
\begin{align*}
&\|\psi_{exact}-\psi_{appr}\|_{\F}:=\|e^{i t \H}e^{-\sqrt{N}
\A(\phi_{0})}e^{-\B(k(0))}\Omega-
e^{i \chi(t)}e^{-\sqrt{N}\A(\phi(t))}e^{-\B(k(t))}\Omega\|_{\F}\\
&\le \frac{C}{N^{\frac{1}{6} }}
\end{align*}
for $0 \le t \le T_0$.
\end{theorem}

Several remarks are in order.

\begin{remark}  First, we comment on the initial conditions for our equations.
The kernel considered in \cite{B-O-S} (for  $\beta=1$ ) is of the form
 $k(t, x, y) =- N \phi(t, x) \phi(t, y)
w (N(x-y))$ where $\phi$ solves the Gross-Pitaevskii NLS and $1-w$ is the solution to the zero energy scattering solution.
The function $w(x)$ is smooth near 0 and behaves like
$\frac{a}{|x|}$ at infinity. This is close to the expected form of the ground state and corresponds, in our set-up, to
\begin{align*}
\Lambda(0, x, y)=\phi(x)\phi(y)(1-N^{\beta-1}w(N^{\beta} (x-y))
 \end{align*}
 We can prescribe $\Lambda(0, x, y)$ arbitrarily, but the time derivative is determined by the equation \eqref{evLambda2} (see Section
 \eqref{self-contained})
 which has a singular term $\frac{1}{N} v_N$ and
 \begin{align*}
 \frac{1}{i}\partial_{t}\Lambda(t, x, y) =
 -\left(-\Delta_{x}-\Delta_{y}+\frac{1}{N}v_N(x-y)\right)\Lambda(t, x, y)  + \mbox{ smoother terms}
 \end{align*}
  Our initial conditions \eqref{IVP3} with  one time derivative are compatible with such initial condition provided
 \begin{align*}
 &<\nabla_x>^{\frac{1}{2}+\epsilon_0} <\nabla_y>^{\frac{1}{2}+\epsilon_0}\left(-\Delta_{x, y}+\frac{1}{N} v_N(x-y)\right)\phi(x)\phi(y)(1-N^{\beta-1}w(N^{\beta} (x-y))\\
 &\in L^2(dx dy)
 \end{align*}
which is true  if, for instance,
 \begin{align*}
 \left(-\Delta_{x, y}+\frac{1}{N} v_N(x-y)\right)(1-N^{\beta-1}w(N^{\beta} (x-y))=0
 \end{align*}
 but are incompatible with the choice $k=0$, $
\Lambda(0, x, y)=\phi(x)\phi(y)$. This is in contrast to our earlier results \cite{GMM1, GMM2, GM} on the uncoupled equations.  

\end{remark}

\begin{remark}  The value $\beta < 2/3 $  is the highest value  for which the potential $N^{3 \beta -1}v(N^{\beta}(x-y))$ can be treated as a perturbation using Strichartz estimates in various places in the paper (see the next section).  We hope it will be possible to
develop a more refined analysis which does not treat the potential as a perturbation, and
extend the range of $\beta$. 
\end{remark}

\begin{remark}
 Our theorem is stated and proved only locally in (small) time. The proof is based on
Theorem \eqref{mainNL} which is, essentially, a local existence theorem with bounds uniform in $N$ for initial conditions  with
\begin{align*}
&\|<\nabla_x>^{\frac{1}{2}+\epsilon}<\nabla_y>^{\frac{1}{2}+\epsilon}\Lambda(0, \cdot)\|_{L^2}\le C\\
&\|<\nabla_x>^{\frac{1}{2}+\epsilon}<\nabla_y>^{\frac{1}{2}+\epsilon}\Gamma(0, \cdot)\|_{L^2}\le C\\
&\|<\nabla_x>^{\frac{1}{2}+\epsilon} \phi(0, \cdot)\|_{L^2} \le C
\end{align*}
 However, one can show, based on the conservation laws proved in \cite{GMM}, for $\epsilon$ sufficiently small,
  \begin{align*}
&\|<\nabla_x>^{\frac{1}{2}+\epsilon}<\nabla_y>^{\frac{1}{2}+\epsilon}\Lambda(t, \cdot)\|_{L^2}\le C (t) \\
&\|<\nabla_x>^{\frac{1}{2}+\epsilon}<\nabla_y>^{\frac{1}{2}+\epsilon}\Gamma(t, \cdot)\|_{L^2}\le C\\
&\|<\nabla_x>^{\frac{1}{2}+\epsilon} \phi(t, \cdot)\|_{L^2} \le C
\end{align*}
  uniformly in $N$. This makes it likely that our main result Theorem \eqref{GM3thm} extends globally
  in time with some constant $C= C( t)$ depending on $t$. We do not yet know  how $C( t)$  will depend on $t$ as $t \to \infty$.

\end{remark}

\begin{remark} In the very recent and important paper \cite{BCS}, Bocatto, Cenatiempo and Schlein prove a   result closely related to ours in the full range $\beta <1$,  and the estimate is global in time.  However, there are substantial differences between their work and ours.
 The techniques used in \cite{BCS} are quite different than ours, and the approximation in
 \cite{BCS}
 is given (translating to our notation) by
 $e^{i \chi(t)}e^{-\sqrt{N}\A(\phi(t))}e^{-\B(k(t))}U_{2, N}(t) \Omega$ where $k(t)=k(t, x, y)$ is explicit (and related but different from our $k(t)$) and  $U_{2, N}(t)$ is an evolution in Fock space with a quadratic generator (see the page preceding Theorem 1.1 in \cite{BCS}). Given the complexity the evolution equation defining of $U_{2, N}(t)$,
 we believe there is still sufficient interest in having an approximation given by just
 $e^{i \chi(t)}e^{-\sqrt{N}\A(\phi(t))}e^{-\B(k(t))} \Omega$ where $k$ satisfies a classical PDE in $6+1$ variables.

\end{remark}

\section{Guide to the proof}\label{guide}

Before going into the details of the complete proof of the Main Theorem \eqref{GM3thm}, we will explain the main ideas.
In order to control the error terms in section
\eqref{fockerror}, we need estimates for $\|\sht\|_{L^2(dx dy)}$, uniformly in $N$. This is accomplished in Section
 \eqref{estforu}, but here is a miniature simplified sketch. ($\S$ denotes $\frac{1}{i} \frac{\partial}{\partial t} - \Delta$
 in $3+1$ or $6+1$ dimensions.)

 Replace equation \eqref{sseq} by the simplified version
\begin{align*}
&\S u=- v_N (x-y) \Lambda(t, x, y) \sim -\delta(x-y) \Lambda(t, x, x)
\end{align*}
( $u$ denotes $\sht$).
The right hand side is too singular to apply energy estimates or Strichartz estimates, so we proceed by writing Duhamel's formula and integrating by parts:
\begin{align}
&u(t, x, y)= u(0, x, y)+\int_0^t e^{i(t-s)\Delta_{x, y}} \delta(x-y) \Lambda(s, x, x)ds \notag \\
&= \int_0^t e^{i(t-s)\Delta_{x, y}}\Delta_{x, y}^{-1} \frac{\partial}{\partial s}\big(\delta(x-y) \Lambda(s, x, x)\big)ds
\label{timetrick}
\\
& +\mbox{boundary terms} \notag
\end{align}
Now $\Delta_{x, y}^{-1}$ smoothes out the singularity of $\delta(x-y)$, but we need good estimates for
\begin{align}
\frac{\partial}{\partial s} \Lambda(s, x, x). \label{whatweneed}
\end{align}
Therefore, we need to assume good estimates for $\frac{\partial}{\partial s} \Lambda(s, x, y)\bigg|_{s=0}$.
Continuing, we derive coupled equations for $\phi$, $\Lambda$ and $\Gamma$, see Theorem
\eqref{equivthm}.
For this discussion, just look at the simplified model:
\begin{align*}
\S \phi(t, x) = &-\int dz \,  v_N(x-z) \Lambda(t, x, z) \bar \phi(t, x) \\
\left(\S +\frac{1}{N} v_N\right) \Lambda(t, x, y) =&- \int dz\, v_N(x-z)\Lambda(t, x, z) \bar \phi(t, z)  \phi(t, y)\\
&- \int dz \,  v_N(y-z)\Lambda(t, y, z) \bar \phi(t, z)  \phi(t, x)
\end{align*}
or, replacing $v_N$ by $\delta$,
\begin{align}
\S \phi(t, x) = &- \Lambda(t, x, x) \bar \phi(t, x) \label{eq2phi}\\
\left(\S +\frac{1}{N} v_N(x-y)\right) \Lambda(t, x, y) =&- \Lambda(t, x, x) \bar \phi(t, x)  \phi(t, y) \label{eq1lambda}\\
& -\Lambda(t, y, y) \bar \phi(t, y)  \phi(t, x)
\notag
\end{align}
It is well known that NLS is well-posed in $H^{1/2}$ in $3+1$ dimensions,
so it is natural to prove a well-posedness result with $\nabla^{1/2} \phi(0, x) \in L^2$ and
$\nabla_x^{1/2}\nabla_y^{1/2} \Lambda(0, x, y) \in L^2$. To get things started, we need the space-time collapsing estimate
\eqref{higherlambda}
of Lemma \eqref{spacetime}. This holds for solutions of the homogeneous Schr\"odinger equations, and, automatically, in $X^{1/2+}$
spaces\footnote{$1/2+$ is a number slightly bigger than $1/2$.} (see section \eqref{estsection} for the definition and properties of these spaces).
Ignoring the potential for a moment, it is very easy to treat the equation
\eqref{eq1lambda}.
\begin{align*}
\S \bigg(\nabla_x^{1/2}\nabla_y^{1/2} \Lambda(t, x, y)\bigg)
=&- \nabla_x^{1/2}\bigg(\Lambda(t, x, x) \bar \phi(t, x) \bigg)\nabla_y^{1/2} \phi(t, y) \\
& -\nabla_y^{1/2}\bigg(\Lambda(t, y, y) \bar \phi(t, y) \bigg)\nabla_x^{1/2} \phi(t, x)
\end{align*}
If $\nabla^{1/2} \phi \in X^{1/2+}$ and $\nabla_x^{1/2}\nabla_y^{1/2}\Lambda \in X^{1/2+}$ then, locally in time,\\
 $\nabla^{1/2} \phi \in L^{\infty}(dt) L^2(dx)$ and  $\nabla_x^{1/2} \Lambda(t, x, x) \in L^2(dt)L^2(dx)$
 by \eqref{higherlambda},
 and thus
\begin{align*}
\nabla_x^{1/2}\left(\Lambda(t, x, x) \bar \phi(t, x) \right)\nabla_y^{1/2} \phi(t, y)&\in L^{2}(dt) L^{6/5}(dx) L^2(dy)\\ &\subset X^{-1/2-}
\end{align*}
and similarly for the second term, which puts $\nabla_x^{1/2}\nabla_y^{1/2} \Lambda(t, x, y)$ back in $X^{1/2+}$
(once we resolve the technical discrepancy between $1/2+$ and $1/2-$ by introducing some epsilons), closing the loop in $X$ spaces. The argument for equation \eqref{eq2phi} is similar.

When applying these estimates to \eqref{eq1lambda} we also have to differentiate the potential term, and
$\nabla_x^{1/2}\nabla_y^{1/2} \frac{1}{N} v_N \in L^{6/5}$ uniformly in $N$ if $\beta \le 2/3$. This is the lowest exponent for which the term can be treated as a perturbation using $L^2(dt) L^{6/5}(d(x-y)) L^2(d(x+y))$ Strichartz estimates. The precise statement is Proposition \eqref{mainXest}.

 The spaces $X^{1/2+}$ come with a fixed cut-off function, and we need to vary the size of the time interval at will
 in order to obtain a contraction out of the argument outlined above, so, for this technical reason,  the non-linear result
Theorem \eqref{mainNL}
 is proved in the spaces \eqref{nlamb}-\eqref{nphi}. Once we have this result, we can differentiate the equations with respect to $t$ and get control over $\N_T(\frac{\partial}{\partial t} \Lambda)$ which in turn controls
 \eqref{whatweneed}.
Finally, in order to control Fock space error terms such as
\begin{align*}
&\left(\S +\frac{1}{N}\sum_{1
\le i<j\le 3} v_N(x_i-x_j) \right) E= \frac{1}{\sqrt N} \phi(x_1)v_N(x_1-x_2) \sh(x_2, x_3)\\
&E(0, \cdot)=0 \notag
\end{align*}
(see \eqref{fockspaceest} ) we refrain from using time derivatives of $\sh$  which would require higher derivatives of $\Lambda$ (see
\eqref{timetrick}). Instead we use Strichartz estimates. The fact that $\frac{1}{\sqrt N}\| v_N \|_{L^{6/5}} \le \frac{C}{N^{1/6}}$
(for $\beta \le 2/3)$) explains the exponent in the statement of the theorem. The complete proof has to include several iterates of this type of argument together with energy estimates in Fock space in order to handle off diagonal terms in our reduced Hamiltonian.
\section{The equations for $\Lambda$ and $\Gamma$ (self-contained derivation)} \label{self-contained}

As already mentioned, it seems difficult to obtain estimates (uniformly in $N$) for $\phi$ and $k$ directly from equations \eqref{sseq},
\eqref{wweq}. The equations seem linear, but the "coefficients"  $v_N \Lambda$,
$v_N \Gamma$
depend on $\sht$, $\cht$. We will proceed
indirectly by
deriving and studying equations for $\Lambda$ and $\Gamma$.
\begin{theorem} \label{equivthm}
The equations of Theorem \eqref{mainthm} are equivalent to
\begin{subequations}
\begin{align}
& \hskip 0.3 in \left\{\frac{1}{i}\partial_{t}-\Delta_{x_1}\right\}\phi(x_{1}) \label{evphi}
=-\int dy\left\{
v_N(x_{1}-y)\Gamma(y,y)\right\}\phi(x_{1})
\\
&-\int dy\left\{v_N(x_{1}-y)\phi(y)\big(\Gamma(y,x_{1})-\overline{\phi}(y)\phi(x_1)\big)
+v_N(x_{1}-y)\overline{\phi}(y)\big(\Lambda(x_{1},y)-\phi(x_1)\phi(y)\big)\right\} \notag\\
 & \notag\\
&  \hskip 0.3 in \left\{\frac{1}{i}\partial_{t}-\Delta_{x_1}-\Delta_{x_2}+\frac{1}{N}v_N(x_{1}-x_{2})\right\}\Lambda(x_1, x_2) \label{evLambda2}\\
& =-\int dy \left\{v_N(x_{1}-y)\Gamma(y,y) + v_N(x_{2}-y)\Gamma(y,y)\right\}\Lambda(x_1, x_2)
\nonumber
\\
&-\int dy\left\{\big(v_N(x_{1}-y)+v_N(x_{2}-y)\big)\Big(\Lambda(x_{1},y)\Gamma(y,x_{2})+
\overline{\Gamma}(x_{1},y)\Lambda(y,x_{2})\Big)
\right\}+
\nonumber
\\
&+2\int dy\left\{\big(v_N(x_{1}-y)+v_N(x_{2}-y)\big)\vert\phi(y)\vert^{2}\phi(x_{1})\phi(x_{2})\right\}\notag\\
& \notag\\
& \hskip 0.3 in \left\{\frac{1}{i}\partial_{t}-\Delta_{x_1}+\Delta_{x_2}\right\}\overline \Gamma(x_1, x_2)\label{evGamma2}\\
&
=
-\int dy\left\{\big(v_N(x_{1}-y)-v_N(x_{2}-y)\big){\Lambda}(x_{1},y)\overline\Lambda(y,x_{2})\right\}+
\notag
\\
&-\int dy\left\{\big(v_N(x_{1}-y)-v_N(x_{2}-y)\big)\Big(\overline\Gamma(x_{1},y)\overline\Gamma(y,x_{2})+\overline\Gamma(y,y)\overline\Gamma(x_{1},x_{2})\Big)
\right\}
\nonumber
\\
&+2\int dy\left\{\big(v_N(x_{1}-y)-v_N(x_{2}-y)\big)\vert\phi(y)\vert^{2}{\phi}(x_{1})\overline\phi(x_{2})\right\} \notag
\end{align}
\end{subequations}
See \eqref{BB1}-\eqref{BB3} for the conceptual meaning of these equations in terms of the density matrices $\L$ defined below.
Also note that $\phi$, $\Lambda$ and $\Gamma$ depend on $t$, but this dependence has been suppressed in the above formulas.

\end{theorem}
While it is easy to prove this by direct calculation, we proceed with a derivation
 which is independent of Theorem \eqref{mainthm}
of our previous paper
 \cite{GMM}.

As in our previous papers \cite{GMM1}-\cite{GMM}, we
consider
\begin{align}
\M&:=e^{-\sqrt{N}\A}e^{-\B}\qquad\qquad  {\rm and \ we\ have\ the\ evolution} \label{defM}
\\
\frac{1}{i}\partial_{t}\M &=\H\M-\M\H_{\rm red}\qquad {\rm and\ of\ course,}\label{evM}
\\
-\frac{1}{i}\partial_{t}\M^{\ast}&=\M^{\ast}\H-\H_{\rm red}\M^{\ast}\ . \label{evMstar}
\end{align}
The evolution equation for $\M$ above is obvious from \eqref{reduced}.

Take a monomial of the form: (Wick ordered)
\begin{align*}
\P_{m,n} =a^{\ast}_{y_1}a^{\ast}_{y_2}\ldots a^{\ast}_{y_m}a_{x_{1}}a_{x_{2}}\ldots a_{x_{n}}
\end{align*}
and define the $\L$ matrices as follows,
\begin{align*}
&\L_{m, n}(t,y_{1},\ldots, y_{m}; x_{1},\ldots, x_{n}):
=\frac{1}{N^{(m+n)/2}}\big<a_{y_1} \cdots a_{y_m} \M\Omega, \,  a_{x_1} \cdots a_{x_n} \M\Omega \big>\\
&=\frac{1}{N^{(m+n)/2}}\big<\Omega, \, \M^{\ast}\P_{m,n}\M\Omega\big>\\
\end{align*}
The notation is chosen so that the second set of variables are  un-barred.
We will often skip the $t$ dependence, since it is passive in the calculations that we have in mind.

Fortunately we will only need $\L_{0, 1}$, $\L_{1,1}$ and $\L_{0, 2 }$ (which turn out to be $\phi$, $\Gamma$ and $\Lambda$) but the computation is quite general.

To get started,
 we observe that from the evolution of the operator $\M$ we have,
\begin{align}
&\frac{1}{i}\partial_{t}\Big(\M^{\ast}\P\M \Big)=\big[\H_{\rm red},\M^{\ast}\P\M\big]+\M^{\ast}\big[\P,\H\big]\M
\qquad\qquad  {\rm hence} \notag
\\
&\frac{1}{i}\partial_{t}\L =\frac{1}{N^{(n+m)/2}}\left(\big<\Omega,\big[\H_{\rm red},\M^{\ast}\P\M\big]\Omega\big> +
\big<\Omega, \M^{\ast}\big[\P,\H\big]\M\Omega\big> \right)  \label{twoeq}\ .
\end{align}

 At this point we record the following lemma
 \begin{lemma} \label{canc}
If $\H_{\rm red}\Omega =\big(\mu ,0,0,X_{3},X_{4},0\ldots\big)$ and
\begin{align*}
\P:=a\qquad {\rm or}\qquad \P:=aa\qquad {\rm or}\qquad \P:=a^{\ast}a
\end{align*}
then
\begin{align}
\big<\Omega,\big[\H_{\rm red},\M^{\ast}\P\M\big]\Omega\big>=0,
\end{align}
leaving only the second term in \eqref{twoeq}

\end{lemma}
\begin{proof}

We use the following notation:
\begin{align*}
&c=\ch, \, \, u=\sh\\
&a_{x}(c):=\int dy\left\{a_{y}c(y,x)\right\} ,\qquad a^{\ast}_{x}(u):=\int dy\left\{a^{\ast}_{y}u(y,x)\right\}
\\
&a^{\ast}_{x}(\overline{c}) :=\int dy\left\{a^{\ast}_{y}\overline{c}(y,x)\right\} =\int dy\left\{c(x,y)a^{\ast}_{y}\right\}
\qquad {\rm by\ symmetry}
\\
&a_{x}(\overline{u}):=\int dy\left\{a_{y}\overline{u}(y,x)\right\}=\int dy\left\{\overline{u}(x,y)a_{y}\right\}
\qquad {\rm by\ symmetry}
\end{align*}

We have the conjugation formulas (see also \eqref{conjf})
\begin{subequations}
\begin{align}
\M^{\ast}a_{x}\M &=a_{x}(c)+a^{\ast}_x(u)+\sqrt{N}\phi(x):=b_{x}+\sqrt{N}\phi(x) \label{conj}
\\
\M^{\ast}a^{\ast}_{x}\M&=a^{\ast}_{x}(\overline{c})+a_{x}(\overline{u})+\sqrt{N}\overline{\phi}(x):=b^{\ast}_{x}+\sqrt{N}
\overline{\phi}(x) \label{conj1}
\end{align}
\end{subequations}
which implies the transformation of the monomial,
\begin{align*}
\M^{\ast}\P(a^{\ast},a)\M=\P\big(b^{\ast}+\sqrt{N}\overline{\phi}\ ,\ b+\sqrt{N}\phi\big)\ .
\end{align*}
Now if $\P=a$ then, using \eqref{abstract},
\begin{align*}
\big<\Omega, \H_{\rm red}\big(b_{x}+\sqrt{N}\phi(x)\big)\Omega \big> -
\big<\Omega,\big(b_{x}+\sqrt{N}\phi(x)\big)\H_{\rm red}\Omega \big> =0\ .
\end{align*}
The argument is similar if $\P=aa$ or $\P=a^{\ast}a$ since only the entries in the zeroth slot survive.

\end{proof}

Based on this, we easily prove the following proposition.

\begin{proposition} \label{maineqq} Under the assumptions of Lemma \eqref{canc}, the following equations hold
\begin{subequations}
\begin{align}
&\left(\frac{1}{i} \frac{\partial}{\partial t}
-\Delta_{x_1}\right)\label{BB1}
\L_{0, 1}(t, x_1)
  \\
&= -  \int v_N(x_1 - x_2) \L_{1, 2} (t,  x_2; x_1,   x_2) dx_2 \notag \\
& \notag\\
&\left(\frac{1}{i} \frac{\partial}{\partial t} \label{BB2}
+\Delta_{x_1}  - \Delta_{y_1} \right)\L_{1, 1}(t, x_1; y_1)
 \\
&=  \int v_N(x_1 - x_2) \L_{2, 2} (t, x_1, x_2; y_1, x_2) dx_2 \notag
 - \int v_N(y_1 - y_2)\L_{2, 2} (t, x_1, y_2; y_1, y_2)dy_2 \notag\\
 & \notag\\
&\left(\frac{1}{i} \frac{\partial}{\partial t} \label{BB3}
-\Delta_{x_1} -\Delta_{x_2}
+\frac{1}{N}  v_N(x_1-x_2)\right)\L_{0, 2}(t, x_1, x_2)\\
&= -
 \int v_N(x_1-y)\L_{1, 3}(t, y; x_1, x_2, y)dy
  -
 \int v_N(x_2-y)\L_{1, 3}(t, y; x_1, x_2, y)dy \notag
\end{align}
\end{subequations}
\end{proposition}
The equation \eqref{BB2} is one of the BBGKY equations, but in our case $\L_{2, 2}$ can be expressed in terms of the earlier matrices, see Lemma \eqref{lmat}.

\begin{proof}
With $\P$ any monomial of degree one or two we know from Lemma \eqref{canc} that
\begin{align*}
\big<\Omega, \big[\H_{\rm red},\M^{\ast}\P\M\big]\Omega \big>=0
\end{align*}
and for any of the corresponding matrices $\L$ we arrive at the equation,
\begin{align*}
\frac{1}{i}\partial_{t}\L=\frac{1}{N^{\alpha}}\big<\Omega, \M^{\ast}\big[\P,\H\big]\M\Omega \big>
\qquad ,\qquad \alpha =1/2\ ,\ 1 \ .
\end{align*}
We need to compute $[\P,\H]$ and for this purpose recall our original Hamiltonian,

\begin{align*}
&\H=\int dxdy\left\{
\Delta_{x}\delta(x-y)a^{\ast}_{x}a_{y}\right\}-\frac{1}{2N}\int dxdy\left\{v_{N}(x-y)a^{\ast}_{x}a^{\ast}_{y}a_{x}a_{x}\right\}
\end{align*}
 Below we list the commutators with each mononomial:
\begin{align*}
\big[a_{x_1},\H\big]&=\Delta_{x_1}a_{x_{1}}-
\frac{1}{N}
\int dy\left\{v_N(x_{1}-y)a^{\ast}_{y}a_{y}\right\}a_{x_1}
\\
\big[a_{x_{1}}a_{x_{2}},\H\big]&=\Delta_{x_1}a_{x_{1}}a_{x_{2}}
+\Delta_{x_2}a_{x_{1}}a_{x_{2}}
-\frac{1}{N}v_N(x_{1}-x_{2})a_{x_{1}}a_{x_{2}}
\\
&-\frac{1}{N}\int dz\left\{\big(v_N(x_{1}-z)+v(x_{2}-z)\big)a^{\ast}_{z}a_{z}a_{x_1}a_{x_2}\right\}
\\
\big[a^{\ast}_{x_1}a_{x_2},\H\big]&=a^{\ast}_{x_1}\Delta_{x_2} a_{x_2}-\left(\Delta_{x_1}a_{x_1}\right)^*a_{x_2}
\\
&+\frac{1}{N}\int dz\left\{\big(v_N(x_{1}-z)-v_N(x_{2}-z)\big)a^{\ast}_{x_1}a^{\ast}_{z}a_{z}a_{x_{2}}\right\}
\end{align*}
from which we can derive the corresponding evolution equations for the $\L$ matrices:
\begin{align}
\left\{\frac{1}{i}\partial_{t}-\Delta_{x_{1}}\right\}\L_{0, 1}(t,x_{1})=
-\frac{1}{N^{3/2}}\int dy v_N(x_1-y)\left\{\big<\Omega, \M^{\ast}a^{\ast}_{y}a_{y}a_{x_1}\M\Omega \big>\right\}
\label{evphi-0}
\end{align}
\begin{align}
&\left\{\frac{1}{i}\partial_{t}-\Delta_{x_{1}}-\Delta_{x_2}+\frac{1}{N}v_N(x_{1}-x_{2})\right\}\L_{0,2}(t,x_{1},x_{2})=
\nonumber
\\
&-\frac{1}{N^{2}}\int dz\left\{\big(v_N(x_{1}-z)+v_N(x_{2}-z)\big)
\big<\Omega, \M^{\ast}a^{\ast}_{z}a_{z}a_{x_1}a_{x_2}\M\Omega \big>\right\}
\label{evLambda-0}
\end{align}
 and finally
\begin{align}
&\left\{\frac{1}{i}\partial_{t}+\Delta_{x_1}-\Delta_{x_2}\right\}\L_{1,1}(t,x_{1},x_{2}) =
\nonumber
\\
&+\frac{1}{N^{2}}\int dz\left\{\big(v_N(x_{1}-z)-v_N(x_{2}-z)\big)\big<\Omega, \M^{\ast}a^{\ast}_{x_1}a^{\ast}_{z}a_{z}a_{x_2}\M\Omega \big>\right\}
\label{evGamma-0}
\end{align}
which implies the statement of the proposition.
\end{proof}
The proof of Theorem  \eqref{equivthm} is finished by computing the necessary matrices $\L$.
We need the following
(writing throughout this proof $[x]$ indicates  freezing the variable $x$):
\begin{align}
&\M^{\ast}a_{x_1}\M\Omega\nonumber\\
&=\big(b_{x_1}+\sqrt{N}\phi(x_{1})\big)\Omega =\Big(\sqrt{N}\phi([x_{1}]),u(y,[x_{1}]),0,0\ldots\Big) \label{arot}\\
&\M^{\ast}a_{x_1}a_{x_{2}}\M\Omega=\nonumber
\\
&\big(b_{x_1}+\sqrt{N}\phi(x_{1})\big)\big(b_{x_{2}}+\sqrt{N}\phi(x_{2})\big)\Omega
=\Big(f_{0},f_{1},f_{2},0,0\ldots\Big) \nonumber
\\
&\qquad {\rm where\ the\ entries\ are:}
\nonumber
\\
&f_{0}([x_{1}],[x_{2}])=N\phi([x_{1}])\phi([x_{2}])+\big(u\circ c))[x_{1}],[x_{2}])=N\Lambda([x_{1}],[x_{2}])
\label{f0}
\\
&f_{1}(y,[x_{1}],[x_{2}])=\sqrt{N}\Big[\phi([x_{1}])u(y,[x_{2}])+\phi([x_{2}])u(y,[x_{1}])\Big]
\label{f1}
\\
&f_{2}(y_{1},y_{2},[x_{1}],[x_{2}])=\frac{1}{\sqrt{2}}\Big[u(y_{1},[x_{1}])u(y_{2},[x_{2}])+u(y_{2},[x_{1}])u(y_{1},[x_{2}])\Big]
\label{f2}
\end{align}
and similarly,
\begin{align}
&\M^{\ast}a^{\ast}_{x_1}a_{x_2}\M\Omega=\nonumber
\\
&\big(b^{\ast}_{x_{1}}+\sqrt{N}\phib(x_{1})\big(\big(b_{x_{2}}+\sqrt{N}\phi(x_{2})\big)\Omega
=\Big(g_{0},g_{1},g_{2},0,0\ldots\Big)
\nonumber
\\
&\qquad {\rm where\ the\ entries\ are :}
\nonumber
\\
&g_{0}[x_{1}],[x_{2}])=N\phib([x_{1}])\phi([x_{2}])+(\overline{u}\circ u)([x_{1}],[x_{2}])=N\Gamma([x_{1}],[x_{2}])
\label{g0}
\\
&g_{1}(y,[x_{1}],[x_{2}])=\sqrt{N}\Big[\phib([x_{1}])u(y,[x_{2}])+\overline{c}(y,[x_{1}])\phi([x_{2}])\Big]
\label{g1}
\\
&g_{2}(y_{1},y_{2},[x_{1}],[x_{2}])=\frac{1}{\sqrt{2}}
\Big[\overline{c}(y_{1},[x_{1}])u(y_{2},[x_{2}])+\overline{c}(y_{2},[x_{1}])u(y_{1},[x_{2}])\Big]
\label{g2}
\end{align}
Based on this we easily compute
\begin{lemma} \label{lmat}
The $\L$ matrices are given by
\begin{align*}
\L_{0, 1}(t,x_{1})&=\frac{1}{\sqrt{N}}\big<\Omega, \big(b_{x_1}+\sqrt{N}\phi(x_{1})\big)\Omega \big> =\phi(x_{1})
\\
\L_{0, 2}(t,x_{1},x_{2})
&=\frac{1}{N}\big<\Omega, \big(b_{x_1}+\sqrt{N}\phi(x_{1})\big)\big(b_{x_2}+\sqrt{N}\phi(x_{2})\big)\Omega\big>\\
& =
\frac{1}{N}(u\circ c)(x_{1},x_{2})+\phi(x_{1})\phi(x_{2})
\nonumber
\\
&=\frac{1}{2N}\psi(t,x_{1},x_{2})+\phi(t,x_{1})\phi(t,x_{2}) =\Lambda
\\
&\qquad {\rm where}\qquad \psi =\sht=2u\circ c 
\\
\L_{1,1}(t,x_{1};x_{2})&=\frac{1}{N}\big(\overline{u}\circ u)(x_{1},x_{2})+\phib(x_{1})\phi(x_{2})
\nonumber
\\
&=\frac{1}{2N}\omega(t,x_{1},x_{2})+\phib(t,x_{1})\phi(t,x_{2}) = \Gamma
\\
&\qquad {\rm where}\qquad \omega :=\chh(2k)-1 =2\overline{u}\circ u\quad ,
\\
\L_{1, 2}(x_1; x_2, x_3)&=\Gamma(x_1, x_2) \phi(x_3) + \frac{1}{2N} \psi (x_3, x_2)\overline \phi(x_1)+
\frac{1}{N} \overline u \circ u (x_1, x_3) \phi(x_2)
\\
\L_{2, 2}(y_1, y_2; x_1, x_2)&=\overline \Lambda(y_1, y_2) \Lambda(x_1, x_2) +
 \frac{1}{N^2} \int \overline f_1(y, y_1, y_2) f_1(y, x_1, x_2)dy\\
&+ \frac{1}{N^2} \int \overline f_2(y, z,  y_1, y_2) f_2(y, z, x_1, x_2) dy dz\\
\L_{1, 3}(y_1; x_1, x_2, x_3)& 
= \overline \Gamma (x_1, y_1) \Lambda(x_2, x_3) + \frac{1}{N^2}
\int  \overline g_1(y, x_1, y_1) f_1(y, x_2, x_3)dy\\
&+\frac{1}{N^2} \int \overline g_2(y, z,  x_1, y_1) f_2(y, z, x_2, x_3) dy dz
\end{align*}
where the integrals in the last two formulas can be trivially expressed in terms of $\psi$ and $\omega$, which in turn can be expressed in terms of $\phi$, $\Lambda$ and $\Gamma$.
\end{lemma}
\begin{proof}
All calculations are straightforward. For instance,
\begin{align*}
&\L_{1, 2}(x_1; x_2, x_3)=\frac{1}{N^{3/2}}\big<\Omega,\M^{\ast}a^{\ast}_{x_1}a_{x_2}a_{x_3}\M\Omega \big>=\frac{1}{N^{3/2}}\big<\Omega, \M^{\ast}a^{\ast}_{x_1}a_{x_2}\M\M^{\ast}a_{x_3}\M\Omega \big>\\
&=\frac{1}{N^{3/2}}\big<\M^{\ast}a^{\ast}_{x_2}a_{x_1}\M\Omega, \M^{\ast}a_{x_3}\M\Omega \big>\\
&=\frac{1}{N}\overline g_0(x_2, x_1) \phi(x_3) + \int \overline g_1(y, x_2, x_1) u(y, x_3) dy\\
&=\Gamma(x_1, x_2) \phi(x_3) + \frac{1}{N} u \circ c (x_3, x_2)\overline \phi(x_1) +\frac{1}{N} \overline u \circ u (x_1, x_3) \phi(x_2)\\
&=\Gamma(x_1, x_2) \phi(x_3) + \frac{1}{2N} \sht (x_3, x_2)\overline \phi(x_1)+
\frac{1}{N} \overline u \circ u (x_1, x_3) \phi(x_2)
\end{align*}
\begin{align*}
&\L_{2, 2}(y_1, y_2; x_1, x_2)=
\frac{1}{N^2}\big<\M^{\ast}a_{y_1}a_{y_2}\M\Omega, \M^{\ast}a_{x_1}a_{x_2}\M\Omega\big>\\
&= \frac{1}{N^2}\overline f_0 (y_1, y_2) f_0(x_1, x_2) + \frac{1}{N^2} \int \overline f_1(y, y_1, y_2) f_1(y, x_1, x_2)dy\\
&+ \frac{1}{N^2} \int \overline f_2(y, z,  y_1, y_2) f_2(y, z, x_1, x_2) dy dz
\end{align*}

\end{proof}

With these ingredients at hand we proceed to write down the evolution of $\L_{1}=\phi$ :
\begin{align*}
&\left\{\frac{1}{i}\partial_{t}-\Delta_{x_1}
+\int dy\left\{
v_N(x_{1}-y)\Gamma(y,y)\right\}\right\}\phi(x_{1})=
\\
&-\frac{1}{2N}\int dy\left\{v_N(x_{1}-y)\phi(y)\omega(y,x_{1})
+v_N(x_{1}-y)\overline{\phi}(y)\psi(x_{1},y)\right\}
\end{align*}
and we can eliminate $\psi$ and $\omega$ by the substitution,
\begin{align*}
\omega &=2N\big(\Gamma-\overline{\phi}\otimes\phi\big)
\\
\psi &=2N\big(\Lambda-\phi\otimes\phi\big)
\end{align*}
so that we have a system involving only $\Lambda$, $\Gamma$ and $\phi$ matrices.

The evolution of $\Lambda$ is given by the expression below:
\begin{align*}
&\left\{\frac{1}{i}\partial_{t}-\Delta_{x_1}-\Delta_{x_2}+\frac{1}{N}v(x_{1}-x_{2})
+\int dy\left\{\left(v_N(x_{1}-y)+v_N(x_{2}-y)\right)\Gamma(y,y)\right\}\right\}\Lambda=
\\
&-\frac{1}{2N}\int dy\left\{\big(v_N(x_{1}-y)+v_N(x_{2}-y)\big)\phi(y)\Big(\omega(y,x_{1})\phi(x_{2})+\omega(y,x_{2})\phi(x_{1})\Big)\right\}
\\
&-\frac{1}{2N}\int dy\left\{\big(v_N(x_{1}-y)+v_N(x_{2}-y)\big)\overline{\phi}(y)\Big(\psi(x_{1},y)\phi(x_{2})+\psi(x_{2},y)\phi(x_{1})\Big)\right\}
\\
&-\frac{1}{4N^{2}}\int dy\left\{\big(v_N(x_{1}-y)+v_N(x_{2}-y)\big)\Big(\psi(x_{1},y)\omega(y,x_{2})+\omega(y,x_{1})\psi(x_{2},y)\Big)
\right\}
\end{align*}

 Finally the evolution of $\Gamma$ is given by:
\begin{align*}
&\left\{\frac{1}{i}\partial_{t}+\Delta_{x_1}-\Delta_{x_2}\right\}\Gamma=
\int dy\left\{\big(v_N(x_{1}-y)-v_N(x_{2}-y)\big)\overline{\Lambda}(x_{1},y)\Lambda(y,x_{2})\right\}+
\\
&\frac{1}{2N}\int dy\left\{\big(v_N(x_{1}-y)-v_N(x_{2}-y)\big)\Big(
\phi(y)\omega(y,x_{2})\overline{\phi}(x_{1})+\overline{\phi}(y)\omega(x_{1},y)\phi(x_{2})\Big)\right\} +
\\
&\frac{1}{2N}\int dy\left\{\big(v_N(x_{1}-y)-v_N(x_{2}-y)\big)\Big(
\vert\phi(y)\vert^{2}\omega(x_{1},x_{2})+\omega(y,y)\overline{\phi}(x_{1})\phi(x_{2})\Big)\right\}+
\\
&\frac{1}{4N^{2}}\int dy\left\{\big(v(x_{1}-y)-v(x_{2}-y)\big)\Big(\omega(x_{1},y)\omega(y,x_{2})+\omega(y,y)\omega(x_{1},x_{2})\Big)
\right\}
\end{align*}

If we substitute $\psi =2N\big(\Lambda-\phi\otimes\phi\big)$ and
$\omega=2N\big(\Gamma-\overline{\phi}\otimes\phi\big)$ we obtain the equations of
 Theorem \eqref{equivthm}.

\begin{remark} Using  the ideas in this section one can easily show that the expected  number of particles
for our approximation,  $\big<\M \Omega,
\mathcal{N} \M\Omega\big>$
(where $\mathcal{N}= \int a_x^*a_x dx$),
as well as the energy $\big<\M \Omega,
\H \M\Omega\big>$ are constant in time. This provides an easier proof of some of the results of section 8 of \cite{GMM}.
\end{remark}

\section{Estimates}\label{estsection}

In this paper, $\S$ and $\S_{\pm}$ stand for the pure differential operators, so that the operators of Theorem \eqref{mainthm}
are
$\tilde{\S} = \S + \, potential \, \, terms$, $\tilde{\W} = \S_{\pm} + \, potential \, \, terms$
\begin{align*}
&\S=\frac{1}{i}\partial_{t} -\Delta\, \, \mbox{ (in $6+1$ or $3+1$ dimensions, as will be clear from the context)}\\
&\S_{\pm}=\frac{1}{i}\partial_{t} -\Delta_x + \Delta_y
\end{align*}
and $potential \, \, terms$ are given by composition with
$v_N(x-y)\Gamma(x, y)$ and multiplication by $v_N* Tr \Gamma$.

The symbol of $\S$ is $\tau +|\xi|^2$ or $\tau  +|\xi|^2  +|\eta|^2$, depending on dimensions, and the symbol of $S_{\pm}$ is
$\tau  +|\xi|^2  -|\eta|^2$.
We will use the following norms:
\begin{align*}
&\|f\|_{X^{\delta}}=\|\left(1+|\tau +|\xi|^2|\right)^{\delta} \hat f(\tau, \xi)\|_{L^2}\\
&\|f\|_{X_S^{\delta}}=\|\left(1+|\tau +|\xi|^2 +|\eta|^2|\right)^{\delta} \hat f(\tau, \xi, \eta)\|_{L^2}\\
&\|f\|_{X_W^{\delta}}=\|\left(1+|\tau +|\xi|^2 -|\eta|^2|\right)^{\delta} \hat f(\tau, \xi, \eta)\|_{L^2}
\end{align*}
and refer to  \cite{taobook}, section 2.6 for their history and properties. Of special importance is the following result:
(proposition 2.12 in \cite{taobook}):
if $\S u =f$, $\delta >0$ and $\chi(t)$ is a fixed $C^{\infty}_0$, cut-off function, then
\begin{align}
\|\chi(t)u\|_{X^{\frac{1}{2}+\delta}} \lesssim_{\chi, \delta} \|u(0, \cdot)\|_{L^2}+  \|f\|_{X^{-\frac{1}{2}+ \delta}}\label{cutoff}
\end{align}

For the rest of the paper, we will use the standard notation $A \lesssim_{\chi, \delta} B$ to mean "there exists a constant $C$ depending of $\chi$ and $\delta$ such that $A \le C B$.

 We will also use freely the general principle that if an $L^p$ estimate holds for solutions to the homogeneous equation, it also holds in $X^{1/2+\delta}$ spaces, see Lemma 2.9 in \cite{taobook}.

 To get started, fix $w \in \mathcal S$ such that $|\hat v| \le \hat w $
 and
  fix $\epsilon>0$ depending on $ \beta<2/3$ so that   $<\nabla_x>^{\frac{1}{2}+\epsilon}<\nabla_y>^{\frac{1}{2}+\epsilon} \frac{1}{N} w_N(x-y)$,
 which is a function of $x-y$,
satisfies
\begin{align*}
  \|<\nabla_x>^{\frac{1}{2}+\epsilon}<\nabla_y>^{\frac{1}{2}+\epsilon} \frac{1}{N} w_N\|_{L^{6/5}(d(x-y)} \le \frac{C}{N^{small \, \, power}}
   \end{align*}

The basic space-time  collapsing estimates, in the spirit of \cite{K-MMM}, are:
\begin{lemma} \label{spacetime}
If $\S \Lambda=0$ then
\begin{align}
&\|| \Lambda(t, x, x) \|_{L^2(dt dx )} \lesssim \||\nabla|_{x-y}^{1/2}  \Lambda_0(x, y)\|_{L^2(dx dy)}
\label{lowerlambda}\\
&\||\nabla|_x^{1/2} \Lambda(t, x, x) \|_{L^2(dt dx )} \lesssim \||\nabla|_x^{1/2} |\nabla|_y^{1/2} \Lambda_0(x, y)\|_{L^2(dx dy)}
\label{higherlambda}
\end{align}
As a consequence, if  $\S \Lambda=F$ and $\delta>0$,  then
\begin{align}
&\sup_z\|<\nabla_x>^{\frac{1}{2}+\epsilon} \chi(t)\Lambda(t, x, x+z) \|_{L^2(dt dx )} \label{higherlambdaX} \\
&\lesssim_{\delta } \|<\nabla_x>^{\frac{1}{2}+\epsilon} <\nabla_y>^{\frac{1}{2}+\epsilon }\chi(t)\Lambda \|_{X_S^{1/2+ \delta}} \notag \\
 &\lesssim_{\delta } \|<\nabla_x>^{\frac{1}{2}+\epsilon}<\nabla_y>^{\frac{1}{2}+\epsilon} \Lambda_0(x, y)\|_{L^2(dx dy)}\notag\\
 &
+ \|<\nabla_x>^{\frac{1}{2}+\epsilon}<\nabla_y>^{\frac{1}{2}+\epsilon } F\|_{X_S^{-1/2+ \delta}} \notag
\end{align}
If
 $\S_{\pm} \Gamma=0$, then
\begin{align}
&\||\nabla_x|^{\frac{1}{2}+\epsilon} \Gamma(t, x, x) \|_{L^2(dt dx)} + \notag
\||\nabla_x|^{1/2} \Gamma(t, x, x) \|_{L^2(dt dx)}\\
&\lesssim_{\epsilon} \|<\nabla_x>^{\frac{1}{2}+\epsilon}<\nabla_y>^{\frac{1}{2}+\epsilon}  \Gamma_0(x, y)\|_{L^2(dx dy)}
\label{gammab}
\end{align}

As a consequence, if
 $\S_{\pm} \Gamma=F$, then
\begin{align}
&\sup_z\||\nabla_x|^{\frac{1}{2}+\epsilon}\chi(t) \Gamma(t, x, x+z) \|_{L^2(dt dx)} \notag
+\sup_z\||\nabla_x|^{1/2}\chi(t) \Gamma(t, x, x+z) \|_{L^2(dt dx)}
\\
&\lesssim_{\epsilon, \delta }<\nabla_x>^{\frac{1}{2}+\epsilon}<\nabla_y>^{\frac{1}{2}+\epsilon }\chi(t) \Gamma \|_{X_W^{1/2+ \delta}}\notag \\
&\lesssim_{\epsilon, \delta } \|<\nabla>_x^{\frac{1}{2}+\epsilon}<\nabla_y>^{\frac{1}{2}+\epsilon}  \Gamma_0(x, y)\|_{L^2(dx dy)}\notag\\
&
+\|<\nabla_x>^{\frac{1}{2}+\epsilon}<\nabla_y>^{\frac{1}{2}+\epsilon} F\|_{X_W^{-1/2+ \delta }}
\label{gammabX}
\end{align}
\end{lemma}
\begin{remark} Notice that the estimates for $\Lambda$ are  different from those for $\Gamma$ at low frequencies. We cannot
estimate $\| \Gamma(t, x, x) \|_{L^2(dt dx)}$. The "lowest" derivative we can estimate
in \eqref{gammab}
is $\||\nabla_x|^{1/2} \Gamma(t, x, x) \|_{L^2(dt dx)}$.
\end{remark}

\begin{proof}
The same  method, inspired by \cite{KM0}, works for both. Let $\widetilde{\Lambda}$ denote the space-time Fourier transform of
$\Lambda$.
For \eqref{lowerlambda}
\begin{align*}
&|\widetilde{\Lambda(t, x, x)}(\tau, \xi)|^2 \\
&\lesssim
 \int \delta(\tau-|\xi-\eta|^2 -|\xi+\eta|^2)\frac{1}{|\eta|}d \eta \\
 &\int \delta(\tau-|\xi-\eta|^2 -|\xi+\eta|^2)|\widehat{\nabla_{x-y}^{\frac{1}{2}} \Lambda_0}(\xi-\eta, \xi+\eta)|^2 d \eta
\end{align*}
In order to prove the estimate, we must show
\begin{align*}
\sup_{\tau, \xi} \int \delta(\tau -|\xi|^2 - |\eta|^2)\frac{1}{|\eta|} d \eta \lesssim 1
\end{align*}
which is obvious.
For \eqref{higherlambda}
\begin{align*}
&|\nabla_x^{1/2}\widetilde{\Lambda(t, x, x)}(\tau, \xi)|^2 \\
&\lesssim
 \int \delta(\tau-|\xi-\eta|^2 -|\xi+\eta|^2)\frac{|\xi|}{|\xi-\eta||\xi+\eta|}d \eta \\
 &\int \delta(\tau-|\xi-\eta|^2 -|\xi+\eta|^2)|\widehat{\nabla_x^{\frac{1}{2}} \nabla_y^{\frac{1}{2}}\Lambda_0}(\xi-\eta, \xi+\eta)|^2 d \eta
\end{align*}
In order to prove the estimate, we must show
\begin{align*}
\sup_{\tau, \xi} \int \delta(\tau -|\xi|^2 - |\eta|^2)\frac{|\xi|}{|\xi-\eta||\xi+\eta|} d \eta \lesssim 1
\end{align*}
Without loss of generality, consider the region $|\xi-\eta| \le |\xi + \eta|$. If $|\xi-\eta| \sim |\xi + \eta|$,
$\frac{|\xi|}{|\xi-\eta||\xi+\eta|} \lesssim \frac{1}{|\eta|}  $ and the integral can be evaluated in polar coordinates.
If $|\xi-\eta| << |\xi + \eta|$ then $|\xi| \sim |\eta|$
Writing $\frac{|\xi|}{|\xi-\eta||\xi+\eta|} \lesssim \frac{1}{|\xi-\eta|} \lesssim \frac{1}{|\eta|\sqrt{1-\cos(\theta)}}$
  where $\theta$ is the angle between $\xi$ and $\eta$, we estimate

\begin{align*}
\sup_{\tau} \int_0^{\pi}\int \delta(\tau - \rho^2)\frac{1}{\rho\sqrt{1-\cos(\theta)}} \rho^2d \rho \sin(\theta) d \theta \lesssim
1
\end{align*}
For \eqref{gammab}
\begin{align*}
&|\widetilde{|\nabla_x|^{\frac{1+\epsilon}{2}}\Gamma(t, x, x)}(\tau, \xi)|= c
|\int \delta(\tau-|\xi-\eta|^2 +|\xi+\eta|^2)|\xi|^{\frac{1+\epsilon}{2}}\hat \Gamma_0(\xi-\eta, \xi+\eta) d \eta|\\
&\lesssim \int \delta(\tau-|\xi-\eta|^2 +|\xi+\eta|^2)\left(|\xi-\eta|^{\frac{\epsilon}{2}}+|\xi+\eta|^{\frac{\epsilon}{2}}\right)
|\xi|^{\frac{1}{2}}
|\hat \Gamma_0(\xi-\eta, \xi+\eta)| d \eta
\end{align*}
Here we use the straightforward estimate
\begin{align*}
I=\sup_{\tau, \xi} \int \delta(\tau - \xi \cdot \eta)\frac{|\xi||\xi\pm\eta|^{\epsilon}}{<\xi-\eta>^{1+\epsilon}<\xi+\eta>^{1+\epsilon}} d \eta \lesssim 1
\end{align*}
To prove this, take $\xi = (|\xi|, 0, 0)$. Then
\begin{align*}
&\int \delta(\tau - \xi \cdot \eta)\frac{|\xi||\xi-\eta|^{\epsilon}}{<\xi-\eta>^{1+\epsilon}<\xi+\eta>^{1+\epsilon}} d \eta\\
&=\int \delta(\tau - |\xi|  \eta_1) \frac{|\xi|}{<\xi-\eta><\xi+\eta>^{1+\epsilon}}d \eta
\lesssim \int_{\mathbb R^2} \frac{1}{<\eta_{2, 3}>^{2+\epsilon}} d \eta_2 d \eta_3 \lesssim 1
\end{align*}

\end{proof}

Next, we record some Strichartz type estimates

\begin{lemma}
The following estimate holds
 \begin{align}
 \|e^{i t (\Delta_x \pm \Delta_y)} f\|_{L^2(dt) L^6(dx) L^2 (d y)} \lesssim \|f\|_{L^2}. \label{Xuwenest}\\
 \|e^{i t (\Delta_x \pm \Delta_y)} f\|_{L^2(dt) L^{\infty}(dx) L^2 (d y)} \lesssim \|<\nabla_x>^{1/2 + \epsilon}f\|_{L^2}. \label{Xuwenestinf}
 \end{align}
 and, as a consequence
 \begin{align}
 \|F\|_{L^2(dt) L^6(dx) L^2 (d y)} \lesssim_{\delta} \|F\|_{X^{1/2+\delta}} \label{XXuwenest}
 \end{align}
 where $X$ can be either $X_S$ or $X_W$.
 \end{lemma}
 \begin{proof}
 We argue as follows:
\begin{align*}
&\|\|e^{i t (\Delta_x \pm \Delta_y)} f\|_{L^2 (d y)}\|_{L^2(dt) L^{6}(dx)}=
\|\|e^{i t \Delta_x } f\|_{L^2 (d y)}\|_{L^2(dt) L^{6}(dx)}\\
&\lesssim \|\|e^{i t \Delta_x } f\|_{L^2(dt) L^{6}(dx)}\|_{L^2 (d y)}\\
&\lesssim \|\|e^{i t \Delta_x } f\|_{L^2(dt) L^{6}(dx)}\|_{L^2 (d y)}
\lesssim \|f\|_{L^2(dx dy)} \, \, \, \mbox{  $3+1$
end-point Strichartz \cite{K-T}}
\end{align*}
The proof of \eqref{Xuwenest} is similar, using Sobolev.
This type of estimate has first appeared, we believe, in \cite{XC1}. Versions of it were used in \cite{C-H1}, \cite{C-H2},
 \cite{elif2}.
\end{proof}

We will need the following refinements. Notice that for $\S$ we can choose $x$, $y$ coordinates or
$x+y$, $x-y$ coordinates, but this is not possible for $\S_{\pm}$.

\begin{lemma}
For each $0<\delta < \frac{1}{2}$ there exists  $ 6/5+>6/5$ a number which can be chosen arbitrarily close to $6/5$ if $\delta$ is small such that
the following estimate holds
\begin{align}
&\|\Lambda\|_{X_S^{-\frac{1}{2}+\delta}} \lesssim_{\delta} \|\Lambda\|_{L^{2}L^{6/5+}(x-y)L^2(x+y)} \label{main}\\
\end{align}
\end{lemma}
\begin{proof}
This is proved by interpolating the estimate dual to \eqref{XXuwenest}
\begin{align}
&\|\Lambda\|_{X_S^{-\frac{1}{2}-\delta}} \lesssim \|\Lambda\|_{L^{2}L^{6/5}(x-y)L^2(x+y)}
\label{dualxuwen}
\end{align}
with
\begin{align}
&\|\Lambda\|_{X_S^{-\frac{1}{2}+\frac{1}{2}}} =\|\Lambda\|_{L^2L^{2}(x-y)L^2(x+y)}  \label{triv}
\end{align}

\end{proof}
Also we will need the closely related
\begin{lemma}
For each $0<\delta < \frac{1}{2}$ there exist numbers $2-$, $6/5+$ arbitrarily close to $2$, $6/5$ if $\delta$ is close to $0$, so that
 \begin{align}
 \|F\|_{X_W^{-1/2+\delta}} \lesssim_{\delta}  \|F\|_{L^{2-}(dt) L^{6/5+}(dx) L^2 d (y)} \label{dual}
 \end{align}
\end{lemma}
\begin{proof} We start with
 \begin{align*}
 &\|F\|_{L^{2+}(dt) L^{6-}(dx) L^2 (d y)} \lesssim \|F\|_{X_W^{1/2+\delta}} \, \, \mbox{(interpolate \eqref{XXuwenest} with energy estimates)}\\
 &\|F\|_{X_W^{-1/2-\delta}} \lesssim  \|F\|_{L^{2-}(dt) L^{6/5+}(dx) L^2 (d y)} \, \, \mbox{(dual estimates)}
 \end{align*}
 and interpolate with the trivial  estimate
 \begin{align*}
 \|F\|_{X_W^{-1/2+1/2}} =  \|F\|_{L^2(dt) L^{2}(dx dy)}
 \end{align*}
 we get \eqref{dual}.
\end{proof}
Finally, we have one more estimate along the same lines
\begin{lemma} \label{vlemma}
If $\S \Lambda =0$ then
\begin{align*}
\|\Lambda\|_{L^2(dt) L^{\infty}(d(x-y))L^2(d(x+y))} \lesssim  \|<\nabla_x>^{\frac{1}{2}+ \epsilon}<\nabla_y>^{\frac{1}{2}+ \epsilon}\Lambda_0\|_{L^2(dx dy)}
\end{align*}
and, as a consequence,
if $\S \Lambda =F$ then
\begin{align*}
&\|\chi(t)\Lambda\|_{L^2(dt) L^{\infty}(d(x-y))L^2(d(x+y))}\\
& \lesssim  \|<\nabla_x>^{\frac{1}{2}+ \epsilon}<\nabla_y>^{\frac{1}{2}+ \epsilon}\Lambda_0\|_{L^2(dx dy)}
+\|<\nabla_x>^{\frac{1}{2}+ \epsilon}<\nabla_y>^{\frac{1}{2}+ \epsilon}F\|_{X_S^{-\frac{1}{2}+ \delta}}
\end{align*}
\end{lemma}
\begin{proof} Writing \eqref{Xuwenestinf} in $x+y, x-y$ coordinates
\begin{align*}
&\|e^{i t (\Delta_x + \Delta_y)} \Lambda_0\|_{L^2(dt) L^{\infty}(d(x-y))L^2(d(x+y))}\\
&\lesssim \|<\nabla_{x-y}>^{\frac{1}{2} + \epsilon}\Lambda_0\|_{L^2(dx dy)}
\lesssim \|<\nabla_x>^{\frac{1}{2} + \epsilon}<\nabla_y>^{\frac{1}{2} + \epsilon}\Lambda_0\|_{L^2(dx dy)}
\end{align*}
\end{proof}
Note that here we are forced to use $<\nabla>$ rather than $\nabla$.

With the help of Lemma \eqref{vlemma}  we can estimate solutions to
\begin{align}
 &\left(\S + \frac{1}{N} v_N(x-y)\right) \Lambda=F \label{veq}\\
 &\Lambda(0, \cdot)=\Lambda_0(\cdot) \notag
\end{align}
The next proposition is the key estimate of our paper.
\begin{proposition} \label{mainXest}
If $\Lambda$ satisfies \eqref{veq}, $\beta <2/3$, and $\chi(t)$ is a smooth cut-off function which is $1$ on $[0, 1]$,  $\delta$ is sufficiently small, and $N\ge N_0$ is sufficiently large, then
\begin{align}
 &\|<\nabla_x>^{\frac{1}{2}+\epsilon} <\nabla_y>^{\frac{1}{2}+\epsilon} \chi(t)\Lambda\|_{X_S^{\frac{1}{2}+\delta}} \label{Xest}\\
 & \lesssim
  \|<\nabla_x>^{\frac{1}{2}+\epsilon} <\nabla_y>^{\frac{1}{2}+\epsilon} \Lambda_0(x, y)\|_{L^2(dx dy)}
+ \|<\nabla_x>^{\frac{1}{2}+\epsilon} <\nabla_y>^{\frac{1}{2}+\epsilon} F\|_{X_S^{-\frac{1}{2}+\delta}}\notag
\end{align}
and thus, by Lemma \eqref{spacetime},
\begin{align}
&\sup_z\|<\nabla_x>^{\frac{1}{2}+\epsilon} \chi(t)\Lambda(t, x+z, x) \|_{L^2(dt dx)}\label{higherlambdaV}\\
 &\lesssim \|<\nabla_x>^{\frac{1}{2}+\epsilon} <\nabla_y>^{\frac{1}{2}+\epsilon} \Lambda_0(x, y)\|_{L^2(dx dy)}
+ \|<\nabla_x>^{\frac{1}{2}+\epsilon} <\nabla_y>^{\frac{1}{2}+\epsilon} F\|_{X_S^{-\frac{1}{2}+\delta}}
\notag
\end{align}
Similar estimates hold for $\left(\frac{\partial}{\partial t}\right)^i\nabla_{x+y} ^j\left(\chi(t)\Lambda\right)$ because these derivatives commute with the potential.

\end{proposition}
\begin{proof}
Recall we introduces a potential $w \in \mathcal S$ such that $\hat w \ge |\hat v|$.
For \eqref{Xest}, we use another cut-off function $\chi_1(t)$ which is identically $1$ on the support of $\chi$ and notice that the solution of \eqref{veq} agrees with that of
\begin{align}
 &\S \Lambda + \frac{1}{N} v_N\chi_1(t) \Lambda=F \label{cuteq}\\
 &\Lambda(0, \cdot)=\Lambda_0(\cdot) \notag
\end{align}
on the support of $\chi$. So we work with the solution of \eqref{cuteq}.
Also, for technical reasons, we define $\lfloor\chi_1(t)\Lambda\rfloor$ to be the inverse Fourier transform of the absolute value of the (space-time) Fourier transform of $\chi_1(t)\Lambda$.
Then
\begin{align}
&\|<\nabla_x>^{\frac{1}{2}+\epsilon} <\nabla_y>^{\frac{1}{2}+\epsilon}\chi_1(t)\Lambda\|_{X_S^{\frac{1}{2}+\delta}} \label{LHS} \\
 &\lesssim \|<\nabla_x>^{\frac{1}{2}+\epsilon} <\nabla_y>^{\frac{1}{2}+\epsilon}\left( \frac{1}{N} v_N \chi_1(t)\Lambda \right)\|_{X_S^{-\frac{1}{2}+\delta}} \label{xterm}\\
  &+\|<\nabla_x>^{\frac{1}{2}+\epsilon} <\nabla_y>^{\frac{1}{2}+\epsilon} \Lambda_0(x, y)\|_{L^2(dx dy)} \notag
+ \|<\nabla_x>^{\frac{1}{2}+\epsilon} <\nabla_y>^{\frac{1}{2}+\delta} F\|_{X_S^{-\frac{1}{2}+\delta}}
\end{align}
We plan to absorb \eqref{xterm} in the LHS of \eqref{LHS}.
The reason we introduced $w$ and $ \lfloor\chi_1(t)\Lambda\rfloor$, which have non-negative Fourier transforms, is to have the following cheap  substitute for the
 Leibnitz rule:
 \begin{subequations}
\begin{align}
&\|<\nabla_x>^{\frac{1}{2}+\epsilon} <\nabla_y>^{\frac{1}{2}+\epsilon}\left( \frac{1}{N}\chi_1(t) v_N \Lambda \right)\|_{X_S^{-\frac{1}{2}+\delta}}\notag\\
&
\lesssim \|\left(<\nabla_x>^{\frac{1}{2}+\epsilon} <\nabla_y>^{\frac{1}{2}+\delta} \frac{1}{N} w_N \right)\lfloor\chi_1(t)\Lambda\rfloor \|_{X_S^{-\frac{1}{2}+\delta}} \label{term1}\\
&+\|\left(<\nabla_x>^{\frac{1}{2}+\epsilon}  \frac{1}{N} w_N\right)  <\nabla_y>^{\frac{1}{2}+\epsilon}\lfloor\chi_1(t)\Lambda\rfloor\|_{X_S^{-\frac{1}{2}+\delta}}\label{term2}\\
&+\|\left(<\nabla_y>^{\frac{1}{2}+\epsilon}  \frac{1}{N} w_N\right)  <\nabla_x>^{\frac{1}{2}+\epsilon}\lfloor\chi_1(t)\Lambda\rfloor\|_{X_S^{-\frac{1}{2}+\delta}}\label{term3}\\
&+\|  \frac{1}{N} w_N<\nabla_x>^{\frac{1}{2}+\epsilon}  <\nabla_y>^{\frac{1}{2}+\epsilon}\lfloor\chi_1(t)\Lambda\rfloor\|_{X_S^{-\frac{1}{2} + \delta}}\label{term4}
\end{align}
\end{subequations}
For the most singular term,  \eqref{term1}, $<\nabla_x>^{\frac{1}{2}+\epsilon}<\nabla_y>^{\frac{1}{2}+\epsilon} \frac{1}{N} w_N(x-y)$ is a function of $x-y$ and we recall
 $\epsilon$ was chosen so that\\
  $\|<\nabla_x>^{\frac{1}{2}+\epsilon}<\nabla_y>^{\frac{1}{2}+\epsilon} \frac{1}{N} w_N\|_{L^{6/5}} \le \frac{C}{N^{small \, \, power}}$. At this stage we insist $\delta$ is so small that the corresponding number $6/5+$ also satisfies\\
  $\|<\nabla_x>^{\frac{1}{2}+\epsilon}<\nabla_y>^{\frac{1}{2}+\epsilon} \frac{1}{N} w_N\|_{L^{6/5+}} \le \frac{C}{N^{small \, \, power}}$.
We estimate, using \eqref{main} and Lemma \eqref{vlemma}
\begin{align*}
&\|\left(<\nabla_x>^{\frac{1}{2}+\epsilon}<\nabla_y>^{\frac{1}{2}+\epsilon}  \frac{1}{N} w_N\right) \lfloor \chi_1(t)\Lambda\rfloor\|_{X^{-\frac{1}{2}+\delta}}\\
&\lesssim\|\left(\nabla_x^{\frac{1}{2}+\epsilon}\nabla_y^{\frac{1}{2}+\epsilon}  \frac{1}{N} w_N\right)\lfloor \chi_1(t)\Lambda\rfloor\|_{L^2(dt) L^{6/5+}(d(x-y)) L^2(d(x+y))}\\
&\lesssim \| \nabla_x^{\frac{1}{2}+\epsilon}\nabla_y^{\frac{1}{2}+\epsilon}  \frac{1}{N} w_N\|_{L^{6/5+}(d(x-y))} \| \lfloor\chi_1(t)\Lambda\rfloor\|_{L^2(dt) L^{\infty}(d(x-y)) L^2(d(x+y))}\\
&\lesssim \frac{1}{N^{small \, \, power}}\,  \| <\nabla_x>^{\frac{1}{2}+\epsilon}<\nabla_y>^{\frac{1}{2}+\epsilon}\lfloor\chi_1(t)\Lambda\rfloor\|_{X^{\frac{1}{2}+\delta}}\\
&= \frac{1}{N^{small \, \, power}}\,  \| <\nabla_x>^{\frac{1}{2}+\epsilon}<\nabla_y>^{\frac{1}{2}+\epsilon}\chi_1(t)\Lambda\|_{X^{\frac{1}{2}+\delta}}\\
\end{align*}
This can be absorbed in the LHS of \eqref{LHS}.
The other terms are easier.

\end{proof}

While we will use the $X$ spaces as tools in our proofs,  the actual norms in which we prove well-posedness  are
$L^p$ norms defined for some $0 < T \le 1$.
\begin{align}
\N_T(\Lambda) =& \sup_z\|<\nabla_x>^{\frac{1}{2}+\epsilon} \Lambda(t, x+z, x) \|_{L^2([0, T] \times \mathbb R^3)}\label{nlamb}\\
&+\|<\nabla_x>^{\frac{1}{2}+\epsilon}<\nabla_y>^{\frac{1}{2}+\epsilon} \Lambda(t, x, y) \|_{L^{\infty}([0, T] \times L^2(\mathbb R^6))}\notag\\
\dot \N_T(\Gamma) =& \label{ngam}\sup_z\||\nabla_x|^{\frac{1}{2}+\epsilon} \Gamma(t, x+z, x) \|_{L^2([0, T] \times \mathbb R^3)}\\
&+\sup_z\|||\nabla_x|^{1/2} \Gamma(t, x+z, x) \|_{L^2([0, T] \times \mathbb R^3)}\notag\\
&+\|<\nabla_x>^{\frac{1}{2}+\epsilon}<\nabla_y>^{\frac{1}{2}+\epsilon} \Gamma(t, x, y) \|_{L^{\infty}([0, T] \times L^2(\mathbb R^6))}\notag\\
\N_T(\phi) =\label{nphi}&\|<\nabla_x>^{\frac{1}{2}+\epsilon} \phi\|_{L^{\infty}([0, T])L^2} + \|<\nabla_x>^{\frac{1}{2}+\epsilon} \phi\|_{L^{2}([0, T])L^6}
\end{align}

Based on energy and Strichartz estimates,
the collapsing estimates of Lemma \eqref{spacetime} and general properties of $X$ spaces which allow one to transfer estimates for solutions to the homogeneous equation to estimates in $X^{1/2+\delta}$ spaces  we have
 \begin{align*}
 &\N_T(\Lambda)\lesssim \|<\nabla_x>^{\frac{1}{2}+\epsilon}<\nabla_y>^{\frac{1}{2}+\epsilon} \Lambda\|_{X_S^{\frac{1}{2} + \delta}}\\
 &\dot \N_T(\Gamma)\lesssim \|<\nabla_x>^{\frac{1}{2}+\epsilon}<\nabla_y>^{\frac{1}{2}+\epsilon} \overline \Gamma\|_{X_W^{\frac{1}{2} + \delta}}\\
 &\N_T(\phi \otimes  \phi) + \dot \N_T(\phi \otimes \overline \phi)
 \lesssim \|<\nabla_x>^{\frac{1}{2}+\epsilon}\phi\|^2_{X^{\frac{1}{2} + \delta}}\\
 \end{align*}

Our basic estimates for the inhomogeneous linear equations are
\begin{proposition} \label{maincorW}
If $\Gamma$ and $\phi$ satisfy
\begin{align*}
 &\S_{\pm}  \Gamma=F \\
 &\S \phi =G
\end{align*}
  $0\le T \le 1$, and $2-<2$, $6/5+>6/5$ are fixed numbers which can be chosen close to $2$ and $6/5$,  then

\begin{align*}
&\dot \N_T(\Gamma)\\
 &\lesssim \|<\nabla_x>^{\frac{1}{2}+\epsilon} <\nabla_y>^{\frac{1}{2}+\epsilon} \Gamma(0, x, y)\|_{L^2(dx dy)}\\
&+ \|<\nabla_x>^{\frac{1}{2}+\epsilon} <\nabla_y>^{\frac{1}{2}+\epsilon} F\|_{L^{2-}[0, T]L^{6/5+}(dx)L^2(dy)}\\
&\N_T(\phi)\\
 &\lesssim \|<\nabla_x>^{\frac{1}{2}+\epsilon} \phi(0, \cdot)\|_{L^2(dx )}
+ \|<\nabla_x>^{\frac{1}{2}+\epsilon}G\|_{L^{2-}[0, T]L^{6/5+}(dx)}\\
&\N_T(\phi \otimes  \phi) + \dot \N_T(\phi \otimes \overline \phi)
 \lesssim \N_T(\phi )^2
\end{align*}
Similar estimates hold when $x$ and $y$ are reversed.
\end{proposition}
\begin{proof}
We will prove this for the equation for $\Gamma$, the other one for $\phi$ being easier.
The standard energy estimate and the collapsing estimate
of lemma \eqref{spacetime} prove that if $\S_{\pm} \Gamma =0$, then
\begin{align*}
\dot \N_T(\Gamma)  \lesssim \|<\nabla_x>^{\frac{1}{2}+\epsilon} <\nabla_y>^{\frac{1}{2}+\epsilon} \Gamma_0(x, y)\|_{L^2(dx dy)}
\end{align*}
From here and  general properties of $X^{1/2+\delta}$ spaces,  we get, for any $\delta>0$,
\begin{align*}
\dot \N_T(\Gamma)
 &\lesssim_{\delta}
 \|<\nabla_x>^{\frac{1}{2}+\epsilon} <\nabla_y>^{\frac{1}{2}+\epsilon} \Gamma\|_{X_W^{1/2+\delta}}
\end{align*}
uniformly in $T$.

 Finally, to prove the stated estimate,
let $F_T=F$ in $[0, T]$, and $0 $ otherwise. Let $\Gamma_T$ be the solution to
\begin{align*}
 &\S_{\pm}  \Gamma_T=F_T\\
 &\Gamma_T(0, \cdot)=\Gamma(0, \cdot)
 \end{align*}
 Then $\Gamma=\Gamma_T$ in $[0, T]$ and recalling $0<T \le 1$ and $\chi=1$ on $[0, 1]$,
\begin{align*}
&\dot \N_T(\Gamma) = \dot \N_T(\chi(t)\Gamma_T)\\
 &\lesssim  \|<\nabla_x>^{\frac{1}{2}+\epsilon}<\nabla_y>^{\frac{1}{2}+\epsilon}\chi(t) \Gamma_T \|_{X_S^{\frac{1}{2} + \delta}}
 \\
 & \lesssim
  \|<\nabla_x>^{\frac{1}{2}+\epsilon} <\nabla_y>^{\frac{1}{2}+\epsilon} \Gamma_0(x, y)\|_{L^2(dx dy)}\\
&+ \|<\nabla_x>^{\frac{1}{2}+\epsilon} <\nabla_y>^{\frac{1}{2}+\epsilon} F_T\|_{X_S^{-\frac{1}{2}+\delta}}
\, \, \mbox{ (by  \eqref{cutoff})}\\
&\lesssim \|<\nabla_x>^{\frac{1}{2}+\epsilon} <\nabla_y>^{\frac{1}{2}+\epsilon} \Gamma_0(x, y)\|_{L^2(dx dy)} \notag \\
&+\|<\nabla_x>^{\frac{1}{2}+\epsilon} <\nabla_y>^{\frac{1}{2}+\epsilon} F_T\|_{L^{2-}(dt)L^{6/5+}(dx)L^2(dy)}  \mbox{ (by \eqref{dual}) }\\
&=
\|<\nabla_x>^{\frac{1}{2}+\epsilon} <\nabla_y>^{\frac{1}{2}+\epsilon} \Gamma_0(x, y)\|_{L^2(dx dy)}\\
&+
\|<\nabla_x>^{\frac{1}{2}+\epsilon} <\nabla_y>^{\frac{1}{2}+\epsilon} F\|_{L^{2-}[0, T]L^{6/5+}(dx)L^2(dy)}
\end{align*}

\end{proof}
The same type of result holds for the equation for $\Lambda$.

\begin{proposition} \label{maincor}

There exist numbers $2-<2$, $6/5+>6/5$, which can be chosen arbitrarily close to
$2$ and $6/5$ such that,  if $\Lambda$ satisfies
\begin{align}
 &\left(\S + \frac{1}{N} v_N(x-y)\right) \Lambda=F \label{veq1}\\
 &\Lambda(0, \cdot)=\Lambda_0(\cdot) \notag
\end{align}
$\beta <2/3$, and $0\le T \le 1$, then

\begin{align}
&\N_T(\Lambda)\label{higherlambdalp}\\
 &\lesssim \|<\nabla_x>^{\frac{1}{2}+\epsilon} <\nabla_y>^{\frac{1}{2}+\epsilon} \Lambda_0(x, y)\|_{L^2(dx dy)}\notag\\
&+ \|<\nabla_x>^{\frac{1}{2}+\epsilon} <\nabla_y>^{\frac{1}{2}+\epsilon} F\|_{L^{2-}[0, T]L^{6/5+}(dx)L^2(dy)}
\notag
\end{align}
Similar estimates hold for $\left(\frac{\partial}{\partial t}\right)^i\nabla_{x+y} ^j\Lambda$:
\begin{align}
&\N_T\left(\left(\frac{\partial}{\partial t}\right)^i\Lambda\right)\label{higherlambdalpd}\\
 &\lesssim \|<\nabla_x>^{\frac{1}{2}+\epsilon} <\nabla_y>^{\frac{1}{2}+\epsilon}
 \left(\frac{\partial}{\partial t}\right)^i\nabla_{x+y} ^j
  \Lambda(t, x, y)\bigg|_{t=0}\|_{L^2(dx dy)}\notag\\
&+ \|<\nabla_x>^{\frac{1}{2}+\epsilon} <\nabla_y>^{\frac{1}{2}+\epsilon}\left(\frac{\partial}{\partial t}\right)^i\nabla_{x+y} ^jF\|_{L^{2-}[0, T]L^{6/5+}(dx)L^2(dy)}
\notag
\end{align}
\end{proposition}
\begin{proof}
The proof is the same as the one of Proposition \eqref{maincorW}, except that, because of the potential $v_N$ we don't get the estimate
\eqref{cutoff} from general principles but use
 Proposition \eqref{mainXest}.
To prove \eqref{higherlambdalpd}, first differentiate the equation, noting that these derivatives commute with the potential, then
apply the previous result.

\end{proof}

\section{The non-linear equations}
Now we come to our main PDE result
\begin{theorem} \label{mainNL}
Let  $\Lambda$, $\Gamma$ and $\phi$ be solutions of
\eqref{evLambda2}, \eqref{evGamma2}, \eqref{evphi} with initial conditions $\phi_0, k_0 \in \mathcal S$.  There exists $N_0$ such that for all $N \ge N_0$,
the following estimates hold:
\begin{align*}
&\N_T\left(\Lambda\right)
\lesssim
\|<\nabla_x>^{\frac{1}{2}+\epsilon}<\nabla_y>^{\frac{1}{2}+\epsilon}\Lambda(0, \cdot)\|_{L^2}\\
&+ T^{small power} \bigg( \N_T\left(\Lambda\right)
 \dot \N_T\left(\Gamma\right)+ \N^4_T\left(\phi\right)\bigg)
\end{align*}
\begin{align*}
&\dot \N_T\left(\Gamma\right)
\lesssim
\|<\nabla_x>^{\frac{1}{2}+\epsilon}<\nabla_y>^{\frac{1}{2}+\epsilon}\Gamma(0, \cdot)\|_{L^2}\\
&+ T^{small power} \bigg( \N^2_T\left(\Lambda\right)+
\dot \N^2_T\left(\Gamma\right)+ \N^4_T\left(\phi\right)\bigg)
\end{align*}

\begin{align*}
&\N_T\left(\phi\right)
\lesssim
\|<\nabla_x>^{\frac{1}{2}+\epsilon}\phi(0, \cdot)\|_{L^2}\\
&+  T^{small power} \bigg(  \N_T\left(\Lambda\right) +  \dot \N_T\left(\Gamma\right)
+\N^2_T\left(\phi\right)\bigg)\N_T\left(\phi\right)
\end{align*}
So there exists $T_0 >0$ such that,
if $T \le T_0$,
\begin{align*}
&\N_T\left(\Lambda\right)+\dot \N_T\left(\Gamma\right)+\N_T\left(\phi\right)\\
&\lesssim
\|<\nabla_x>^{\frac{1}{2}+\epsilon}<\nabla_y>^{\frac{1}{2}+\epsilon}\Lambda(0, \cdot)\|_{L^2}+
\|<\nabla_x>^{\frac{1}{2}+\epsilon}<\nabla_y>^{\frac{1}{2}+\epsilon}\Gamma(0, \cdot)\|_{L^2}\\
&+
\|<\nabla_x>^{\frac{1}{2}+\epsilon}\phi(0, \cdot)\|_{L^2}
\end{align*}

Also, similar estimates hold for the derivatives which commute with the potential:
\begin{align}
&\N_T\left(\frac{\partial}{\partial t}   \nabla_{x+y} ^j\Lambda \right) +\dot \N_T\left(\frac{\partial}{\partial t} \nabla_{x+y} ^j\Gamma\right)
+\N_T\left(\frac{\partial}{\partial t} \nabla_{x} ^j\phi\right)
 \label{commute}\\
&\lesssim \notag
\|<\nabla_x>^{\frac{1}{2}+\epsilon}<\nabla_y>^{\frac{1}{2}+\epsilon}\frac{\partial}{\partial t} \nabla_{x+y} ^j
\Lambda \bigg|_{t=0}\|_{L^2}\notag\\
&+\notag
\|<\nabla_x>^{\frac{1}{2}+\epsilon}<\nabla_y>^{\frac{1}{2}+\epsilon}
\frac{\partial}{\partial t} \nabla_{x+y} ^j
\Gamma \bigg|_{t=0}\|_{L^2}\\
&+\notag
\|<\nabla_x>^{\frac{1}{2}+\epsilon}
\frac{\partial}{\partial t}\nabla_{x} ^j
\phi |_{t=0}\|_{L^2}
\end{align}
The time interval $T_0$ and the implicit constants in the above inequalities depend only on
\begin{align*}
&\|<\nabla_x>^{\frac{1}{2}+\epsilon}<\nabla_y>^{\frac{1}{2}+\epsilon}\Lambda(0, \cdot)\|_{L^2}+\\
&\|<\nabla_x>^{\frac{1}{2}+\epsilon}<\nabla_y>^{\frac{1}{2}+\epsilon}\Gamma(0, \cdot)\|_{L^2}+\\
&\|<\nabla_x>^{\frac{1}{2}+\epsilon} \phi(0, \cdot)\|_{L^2}
\end{align*}

\end{theorem}
\begin{proof}

 For equation \eqref{evphi}, which can be abbreviated as
 \begin{align*}
 &\frac{1}{i}\partial_{t}\phi -\Delta\phi \\
&=- (v_N \Lambda) \circ \overline \phi
-(v_N \overline\Gamma )\circ \phi -( v_N *Tr \Gamma )\cdot \phi \notag
+ 2(v_N*|\phi|^2)\phi\\
  &:=RHS\eqref{evphi}
\end{align*}
we first apply Proposition \eqref{maincorW}:
 \begin{align*}
&  \N_T(\phi) \lesssim
\|<\nabla_x>^{\frac{1}{2}+\epsilon} \phi(0, \cdot)\|_{L^2}+
\|<\nabla_x>^{\frac{1}{2}+\epsilon}RHS\eqref{evphi}\|_{L^{2-}([0, T]) L^{6/5+}(dx)} \\
&\lesssim \|<\nabla_x>^{\frac{1}{2}+\epsilon} \phi(0, \cdot)\|_{L^2}+ T^{small \, power}    \|<\nabla_x>^{\frac{1}{2}+\epsilon}RHS\eqref{evphi}\|_{L^{2}([0, T]) L^{6/5+}(dx)}\\
& \lesssim \|<\nabla_x>^{\frac{1}{2}+\epsilon} \phi(0, \cdot)\|_{L^2}+ T^{small \, power}
 \bigg( \N_T(\Lambda) + \dot\N_T(\Gamma)+ \N^2_T(\phi)\bigg)\N_T(\phi)\\
\end{align*}
The last line follow from  the classical fractional Leibnitz rule in $L^p$ spaces due to Kato and Ponce, \cite{K-Pon} .
 We only present a typical term in RHS\eqref{evphi}:
\begin{align*}
&\|<\nabla_x>^{\frac{1}{2}+\epsilon}\bigg(( v_N *Tr \Gamma )\cdot \phi \bigg)\|_{L^{2}([0, T]) L^{6/5+}(dx)}\\
&\lesssim \| (v_N *Tr \Gamma )\cdot \phi \|_{L^{2}([0, T]) L^{6/5+}(dx)}+
\||\nabla_x|^{\frac{1}{2}+\epsilon}\bigg((v_N *Tr \Gamma )\cdot \phi\bigg)\|_{L^{2}([0, T]) L^{6/5+}(dx)}\\
&\lesssim \|Tr \Gamma\|_{L^2([0, T]) L^3(dx)} \|\phi\|_{L^{\infty}L^{2+}}\\
&+\||\nabla_x|^{\frac{1}{2}+\epsilon}Tr \Gamma \|_{L^{2}([0, T]) L^{2}(dx)}\|\phi\|_{L^{\infty}L^{3+}}
+\|Tr \Gamma \|_{L^{2}([0, T]) L^{3+}(dx)} \|\nabla_x|^{\frac{1}{2}+\epsilon} \phi\|_{L^{\infty}L^{2}}\\
&\lesssim  \dot \N_T(\Gamma)\N_T(\phi)
\end{align*}

Now we deal with the equation \eqref{evLambda2}, which can be abbreviated as
 \begin{align*}
&\left(\S + \frac{1}{N} v_N\right) \Lambda\\
&=-(v_N \Lambda) \circ  \Gamma -  \overline \Gamma \circ (v_N \Lambda)\\
&-\big( \left(v_N*  Tr   \Gamma\right)(x)+ \left(v_N* Tr \Gamma\right)(y)\big)   \Lambda(x, y)\\
&-\left(v_N \overline\Gamma\right)\circ\Lambda - \Lambda\circ \left(v_N\Gamma\right) \\
&+2(v_N*|\phi|^2)(x)\phi (x)\phi(y) +2(v_N*|\phi|^2)(y)\phi (y) \phi(x)\\
&:= RHS\eqref{evLambda2}
\end{align*}
Applying Proposition \eqref{maincor} to the equation \eqref{evLambda2}\\
\begin{align*}
&\N_T(\Lambda)
\lesssim \|<\nabla_x>^{\frac{1}{2}+\epsilon}<\nabla_y>^{\frac{1}{2}+\epsilon} \Lambda(0, \cdot)\|_{L^2}\\
&+\|<\nabla_x>^{\frac{1}{2}+\epsilon}<\nabla_y^{\frac{1}{2}+\epsilon} >RHS\eqref{evLambda2}\|_{L^{2-}[0, T]L^{6/5+}(dx)L^2(dy)}
\end{align*}
We would like to estimate $<\nabla_x>^{\frac{1}{2}+\epsilon}<\nabla_y>^{\frac{1}{2}+\epsilon} RHS\eqref{evLambda2}$ in $L^2L^{\frac{6}{5} + }L^2 $
to gain a small power of $T$ from Cauchy-Schwarz in time.

  Typical term in RHS \eqref{evLambda2}:
 \begin{align*}
  &<\nabla_x>^{\frac{1}{2}+\epsilon}<\nabla_y^{\frac{1}{2}+\epsilon} >\bigg(\left(v_N\Lambda\right) \circ  \Gamma (x, y)\bigg)\\
 & =\int v_N(z)<\nabla_x>^{\frac{1}{2}+\epsilon}<\nabla_y^{\frac{1}{2}+\epsilon} >\left( \Lambda (x, x-z)  \Gamma(x-z, y )\right)dz
 \end{align*}
Applying the fractional Leibniz rule in $L^p$ spaces,
 \begin{align*}
 \int v_N(z) &\|<\nabla_x>^{\frac{1}{2}+\epsilon}<\nabla_y^{\frac{1}{2}+\epsilon} >\bigg( \Lambda (x, x-z)  \Gamma(x-z, y )\bigg)\|_{L^2([0, T])L^{\frac{6}{5} + }(dx)L^2(dy)} dz\\
 \lesssim \int v_N(z)& \|<\nabla_x>^{\frac{1}{2}+\epsilon} \Lambda (x, x-z)\|_{L^2([0, T])L^2(dx)} \times\\
  &\| <\nabla_y^{\frac{1}{2}+\epsilon} >\Gamma(x-z, y )\|_{L^{\infty}([0, T])L^{3+  }(dx)L^2(dy)} dz\\
 +\int v_N(z) &\|\Lambda (x, x-z)\|_{L^2([0, T])L^{3+}(dx)} \times\\
 &\| <\nabla_x>^{\frac{1}{2}+\epsilon}  <\nabla_y^{\frac{1}{2}+\epsilon} >\Gamma(x-z, y )
 \|_{L^{\infty}([0, T])L^{2  }(dx)L^2(dy)} dz\\
 &\lesssim \N_T(\Lambda)\dot  \N_T(\Gamma)
 \end{align*}

 All other terms are treated in a similar manner. In fact, all terms on the RHS\eqref{evLambda2}
  and RHS\eqref{evGamma2}
 are of the form $(v_N F) \circ G$ or $(v_N * Tr F) G$ where $F$, $G$ can be $\Lambda$, $\Gamma$, $\phi \otimes \phi$ of
 $\phi \otimes \overline \phi$ (or their complex conjugates) and $\N_T(F)$, $\N_T(G)$ are estimated as above by
 Proposition \eqref{maincorW} or  Proposition \eqref{maincor}.

\end{proof}

\section{Estimates for $\sht$ } \label{estforu}
Recall the equation \eqref{sseq}, which can be written explicitly as
\begin{align}
&\S\left(\sht\right)= - 2v_N \Lambda\label{sht}\\
&- (v_N \Lambda) \circ p_2 - \overline p_2 \circ (v_N \Lambda)\notag\\
 &-\left(\left(v_N* Tr \Gamma\right)(x)+ \left(v_N* Tr \Gamma\right)(y)\right) \sht(x, y)
-\left(v_N\Gamma\right)\circ \sht - \sht\circ \left(v_N\Gamma\right)\notag\\
&:=RHS\eqref{sht} \notag
 \end{align}
Now that we control the quantities \eqref{commute} we use proofs similar to those of section 4 of \cite{GM}, or section 3
 of \cite{elif2} with $- v_N \Lambda $ playing the role of $m$ and $\tilde \S$ playing the role of $\S$, at least locally in time.
 The crucial ingredient is that
 \begin{align}
 &\sup_z \|\Lambda(t, x +z, x)\|_{L^2(dx)}+ \notag
 \sup_z \|\left(\frac{\partial}{\partial t}\right)^i\Lambda(t, x +z, x)\|_{L^2(dx)}\\
 &+ \|\Lambda\|_{L^{\infty}} + \|\left(\frac{\partial}{\partial t}\right)^i\Lambda\|_{L^{\infty}}
  \lesssim 1 \label{crucial}
  \end{align}
  for $i=0, 1$. The $\nabla_{x+y}$ derivatives have been used to control \\
 $ \|\Lambda(t, x +z, x)\|_{L^{\infty}(dx)} \le C \sum_{j \le 2} \|\left(\nabla^j_{x+y}\Lambda\right)(t, x +z, x)\|_{L^{2}(dx)}$.

 The reader is warned, with apologies, that $\S$ in \cite{GM} is not the $\S$ of the current paper (which is a purely differential operator), but what was called $\S_{old}$ in the introduction.

 In this section, we will prove

 \begin{theorem} \label{s2thm}
Let $\sht$, $\cht$ satisfy the equations \eqref{sseq}, \eqref{wweq} with initial conditions as in Theorem
\eqref{GM3thm}. Then, for $T_0$ as in Theorem
\eqref{mainNL} and $0 \le j \le 2$,
\begin{align}
&\|   \nabla_{x+y}^j  \sht(t, \cdot, \cdot)\|_{L^2(dxdy)} \lesssim 1 \label{l2}\\
&\sup_x\|\sht(t, x, \cdot)\|_{L^2(dy)}  \lesssim  1 \label{linf}
  \, \, \, (0 \le t \le T_0)
\end{align}
\end{theorem}
Define $\ch=\delta(x-y)+p$. The following is immediate, as in the proof of Corollary 4.2 in \cite{GM}:
\begin{corollary} \label{shest}
 The following estimates hold uniformly in $0 \le t \le T_0$
\begin{align*}
&\|     \sh(t, \cdot, \cdot)\|_{L^2(dxdy)} \lesssim 1\\
&\|   p(t, \cdot, \cdot)\|_{L^2(dxdy)} \lesssim  1\\
&\sup_x\|\sh(t, x, \cdot)\|_{L^2(dy)}  \lesssim 1 \\
&\sup_x\|p(t, x, \cdot)\|_{L^2(dy)}  \lesssim 1
\end{align*}

\end{corollary}
We start with some preliminary lemmas.

\begin{lemma} \label{startlemma} (replacing Lemma 4.4 in \cite{GM}.) Let $s_a^0$ be the solution to
\begin{align}
\S s_a^0 = - 2  v_N(x-y) \Lambda \label{starteq}\\
s_a^0(0, x, y)=\sht(0, x, y) \notag
\end{align}
Then
\begin{align*}
&\|  s_a^0(t, \cdot, \cdot)\|_{L^2(dxdy)} \lesssim 1\\
&(0 \le t \le T_0)
\end{align*}
with similar estimates for $\nabla_{x+y}^js_a^0$.

\end{lemma}

This lemma is a particular case (with $F=v_N(x-y)\Lambda(t, x, y)$ of a more general result:
\begin{lemma} \label{ellipest}
Let $F $ be a function of $6+1$ variables ($k \ge 0$, $x, y \in \mathbb R^3$, $t \in [0, 1]$)
and let
\begin{align*}
E(x, y,  t)= \int_0^t e^{i (t-s)\Delta_{x, y}} F(s) ds
\end{align*}
Then
\begin{align*}
&\|E(t, \cdot)\|_{L^2} \\
&\lesssim \sup_{0 \le s \le t}\left(\|F(s, \cdot)\|_{L^2(d(x+y))  L^1(d(x-y))}
+\|\frac{\partial}{\partial s}F(s, \cdot)\|_{L^2(d(x+y))  L^1(d(x-y))}\right)
\end{align*}
\end{lemma}
\begin{proof}
Change variables, so we  work in $x, y$ rather than $x-y, x+y$ coordinates. In these coordinates, integrating by parts,
\begin{align}
&E(t, \cdot)=\int_0^t e^{i(t-s)\Delta_{x, y}} F(s, \cdot)ds \label{low}\\
&= \int_0^t e^{i(t-s)\Delta_{x, y}}\Delta_{x, y}^{-1} \frac{\partial}{\partial s}F(s, \cdot)ds
\label{high}\\
&+e^{it\Delta_{x, y}}\Delta_{x, y}^{-1} F(0, \cdot)
-\Delta^{-1} F(t, \cdot)\notag\
\end{align}
We are going to project in frequencies dual to $x$ only. For the low frequency case we use \eqref{low}:
\begin{align*}
&\|P_{|\xi| \le 1}E(t, \cdot)\|_{L^2}=\|\int_0^t e^{i(t-s)\Delta_{x, y}}P_{|\xi| \le 1} F(s, \cdot)ds\|_{L^2}\\
&\le  \sup_{0 \le s \le t}\|P_{|\xi| \le 1}F(s, \cdot)\|_{L^2(dx dy )}\\
&\lesssim \sup_{0 \le s \le t}\|F(s, \cdot)\|_{L^2(dy ) L^1(dx)}
\end{align*}
where we have used the fact that the (physical space) kernel corresponding to $ P_{|\xi| \le 1}$ is in $L^2(dx)$.
For the high frequency case we use \eqref{high}. We only write down one term, the boundary terms being similar:
\begin{align*}
&\|P_{|\xi| \ge 1}E(t, \cdot)\|_{L^2}=\|\int_0^t e^{i(t-s)\Delta_{x, y}}P_{|\xi| \ge 1}\Delta_{x, y}^{-1} \frac{\partial}{\partial s} F(s, \cdot)ds\|_{L^2}\\
&\le \sup_{0 \le s \le t} \|P_{|\xi| \ge 1}\Delta_{x}^{-1} \frac{\partial}{\partial s} F(s, \cdot)ds\|_{L^2(dxdy)}\\
&\lesssim \sup_{0 \le s \le t}\|\frac{\partial}{\partial s}F(s, \cdot)\|_{L^2(dy ) L^1(dx)}
\end{align*}
where we have used that the kernel of $P_{|\xi| \ge 1}\Delta_{x}^{-1}$ in in $L^2(dx)$.
\end{proof}

Next, we include the potential terms:
\begin{lemma} \label{salemma}
Let $s_a$ be the solution to
\begin{align}
&\tilde \S s_a=- 2v_N \Lambda \label{sa}\\
&s_a(0, x, y)=\sht(0, x, y) \notag
\end{align}
\begin{align*}
&\|  s_a^0(t, \cdot, \cdot)\|_{L^2(dxdy)} \lesssim 1  \, \, \, (0 \le t \le T_0)\\
\end{align*}
with similar estimates for $\nabla_{x+y}^j s_a$.
\end{lemma}
\begin{proof}
Let $V$ be the "potential" part of $\tilde S$:
\begin{align*}
 &V(u)=\left(\left(v_N* Tr \Gamma\right)(x)+ \left(v_N* Tr \Gamma\right)(y)\right) u
+\left(v_N\Gamma\right)\circ u+ u\circ \left(v_N\Gamma\right)
 \end{align*}
 Form the estimates \eqref{commute} for $\Gamma$ we see that $\bigg| \nabla_{x+y}^j   \Gamma\bigg| \lesssim 1$, thus $V$
 and $[V, \nabla_{x+y}^j ]$ have
  bounded operator norm (on $L^2$).
 Write $s_a=s_a^0+s_a^1$ ($s_a^0$ as in the previous lemma) so that $s_a^1$ satisfies the equation
 \begin{align*}
 &\tilde \S s_a^1= -V(s_a^0)\\
 &s_a^1(0, x, y)=0
 \end{align*}
Using energy estimates and the previous lemma we see
\begin{align*}
\|  \nabla_{x+y}^j s_a^1\|_{L^2} \lesssim 1
\end{align*}
 The result follows from the previous lemma.

\end{proof}
Finally, we can prove Theorem \eqref{s2thm}.
\begin{proof} (of theorem \eqref{s2thm}) Write $\sht:=s_2=s_a + s_e$ where $s_a$ satisfies \eqref{sa} and $\cht= \delta(x-y) + p_2$. In analogy with
(64a), (64b) of \cite{GM}, they satisfy
\begin{align*}
\tilde \S(s_e)= -(v_N \Lambda)\circ p_2 - \overline p_2 \circ (v_N \Lambda)\\
\tilde \W(\overline p_2)=-(v_N \Lambda)\circ \overline s_a + s_a \circ (v_N \Lambda)\\
-(v_N \Lambda)\circ \overline s_e + s_e \circ (v_N \Lambda)\\
:=M
-(v_N \Lambda)\circ \overline s_e + s_e \circ (v_N \Lambda)
\end{align*}
Using the result of Lemma \eqref{salemma}, as well as estimates \eqref{commute} for $\Lambda$ we see that
\begin{align*}
\| \nabla_{x+y}^j    M\|_{L^2(dxdy)} \lesssim 1  \, \, \, (0 \le t \le T_0)
\end{align*}
and the result follows by energy estimates.

Finally, to prove \eqref{linf} we will use the $L^2$ estimate
\eqref{l2}:
\begin{align*}
&\big\| \|\sht(x, z)\|_{L^2(dz)}\big\|_{L^{\infty}(dx)}=
\big\| \|\sht(x, x+z)\|_{L^2(dz)}\big\|_{L^{\infty}(dx)}\\
&\le
\big\| \|\sht(x, x+z)\|_{L^{\infty}(dx)}\big\|_{L^2(dz)} \le C \sum_{j=0}^2
\big\| \|\nabla_x^j\left(\sht(x, x+z)\right)\|_{L^{2}(dx)}\big\|_{L^2(dz)}\\
&=C \sum_{j=0}^2 \big\| \|\left(\nabla_{x+y}^j\sht\right)(x, x+z)\|_{L^{2}(dx)}\big\|_{L^2(dz)}\\
&=C\sum_{j=0}^2\|\left(\nabla_{x+y}^j\sht\right)(x, y)\|_{L^{2}(dx dy)} \le C
\end{align*}

\end{proof}

\section{The reduced Hamiltonian \label{reducedham}}
Recall  \eqref{reduced} for the definition of
the reduced Hamiltonian.  $\H_{red}$ was computed, for instance, in Section 5 of \cite{GM}.
We will write it in a different way, using Wick's theorem.
Recall the conjugation formulas
\begin{align}
e^{\B} a_x e^{-\B}& = \label{conjf}
\int \left( \ch (y, x) a_y + \sh (y, x) a^*_y \right) dy\\
&=a(\chb(x, \cdot)) + a^*(\sh(x, \cdot)):=b_{x} \notag\\
e^{\B} a^*_x e^{-\B}&=\int \left( \shb (y, x)a_y + \chb (y, x) a^*_y \right) \notag dy\\
&=a(\shb(x, \cdot))+a^*(\ch(x, \cdot)) \notag
:=b^{\ast}_{x}\ .
\end{align}
The reduced Hamiltonian is
\begin{subequations}
\begin{align}
 &\H_{red}
=N\mu_{0}(t) \label{zeroline}\\
&+N^{1/2}
\int dx\left\{
h(\phi(t,x))b^{\ast}_{x}
+\bar{h}(\phi(t,x))b_x \right\} \label{firstline'}\\
&+
\frac{1}{i}\frac{\partial}{\partial t} \left(e^{\B(k(t))}
\right)e^{-\B(k(t)}
+\int dx  b_x^* \Delta b_x \\
& -\frac{1}{2}
\int dx dy v_N(x-y)\left(\overline \phi(x)\overline \phi(y) b_x b_y
+ \phi(x)\phi(y)  b^*_x b^*_y + 2 \phi(x) \overline \phi(y) b_x^* b_y\right) \label{badquad} \\
&-\int dx( v_N*|\phi|^2) b_x^* b_x\\
 &- \frac{1}{\sqrt N}
\int dx_{1}dx_{2}\left\{
 v_N(x_{1}-x_{2}) \left(\overline{\phi}(x_{2}) b^{\ast}_{x_{1}} b_{x_{1}}b_{x_{2}} + \phi(x_{2})b_{x_{1}}^*b_{x_{2}}^*b_{x_{1}}\right)
\right\}\\
&-\frac{1}{2 N}
\int dx_{1}dx_{2}\left\{ v_N(x_{1}-x_{2}) b^*_{x_1} b^*_{x_2}  b_{x_1}b_{x_2} \right\}\ .
\label{thirdline}
\end{align}
\end{subequations}

The functions $\mu_{0}(t)$ and $h(\phi(t,x))$ appearing in \eqref{zeroline},\eqref{firstline'} are given below,
\begin{align*}
&\mu_{0} :=\int dx\left\{\frac{1}{2i}\big(\phi\bar{\phi}_{t}-\bar{\phi}\phi_{t}\big)-\big\vert\nabla\phi\big\vert^{2}
\right\}
-\frac{1}{2}\int dxdy\left\{v_{N}(x-y)\vert\phi(x)\vert^{2}\vert\phi(y)\vert^{2}\right\}
\\
&h(\phi(t,x)):=-\frac{1}{i}\partial_{t}\phi  +\Delta\phi -\big(v_{N}\ast\vert\phi\vert^{2}\big)\phi\ .
\end{align*}

We re-arrange the terms in $\H_{red}$ using Wick's theorem . Define the contraction of
$A(f):=a(\bar f_1) + a^*(f_2)$ and $A(g):=a(\bar g_1) + a^*(g_2)$ to be $C(A(f), A(g))=[a(\bar f_1), a^*(g_2)]=\int \bar f_1 g_2$, and define
the normal order $\Nor\big(A(f)A(g)A(h)A(k)\big)$ to be $A(f)A(g)A(h)A(k)$ expanded and re-arranged so all starred terms have been moved to the left, as if they commuted with unstarred terms.
Wick's theorem, which can easily be proved by induction, says, in particular, that
\begin{align*}
&A(g)A(h)A(k)= \Nor\big(A(g)A(h)A(k)\big)\\
&  + C(A(g), A(h))A(k)  + C(A(g), A(k))A(h) +  C(A(h), A(k))A(g)
\end{align*}
\begin{align*}
&A(f)A(g)A(h)A(k)\\
&= \Nor\bigg(A(f)A(g)A(h)A(k)\bigg)\\
& + C(A(f), A(g))A(h)A(k) + C(A(f), A(h))A(g)A(k) + \cdots \mbox{ (6 terms)}\\
&=\Nor\bigg(A(f)A(g)A(h)A(k)\\
& + C(A(f), A(g))A(h)A(k) + C(A(f), A(h))A(g)A(k) + \cdots\bigg)\\
& \mbox{ (6 ordered terms with 1 contraction)}\\
&+C(A(f), A(g))C(A(h), A(k)) + C(A(f), A(h))C(A(g), A(k)) + \cdots \\
& \mbox{ (6 ordered terms with 2 contractions)}
\bigg)
\end{align*}
Applying Wick's theorem to $\H_{red}$ (putting the quartic and cubic terms in normal order, but not the quadratics) we get
\begin{subequations}
\begin{align}
 &\H_{red} \notag
=N\mu_{0}(t)  \\
&+N^{1/2}
\int dx\left\{
\tilde h(\phi(t,x))b^{\ast}_{x} \notag
+\bar{\tilde h}(\phi(t,x))b_x \right\}  \\
&+
\frac{1}{i}\frac{\partial}{\partial t} \left(e^{\B(k(t))} \label{quad1}
\right)e^{-\B(k(t)}
+\int dx  b_x^* \Delta b_x \\
& -\frac{1}{2} \label{quad2}
\int dx dy v_N(x-y)\left(\overline\Lambda(x, y) b_x b_y
+ \Lambda(x, y)  b^*_x b^*_y + 2 \overline \Gamma(x, y) b_x^* b_y\right) \\
&-\int dx (v_N*Tr \Gamma) b_x^* b_x  \label{quad3}\\
 &- \frac{1}{\sqrt N}  \notag
\int dx_{1}dx_{2} \Nor \left\{
 v_N(x_{1}-x_{2}) \left(\overline{\phi}(x_{2}) b^{\ast}_{x_{1}} b_{x_{1}}b_{x_{2}} + \phi(x_{2})b_{x_{1}}^*b_{x_{2}}^*b_{x_{1}}\right)
\right\}\\
&-\frac{1}{2 N} \notag
\int dx_{1}dx_{2}\Nor \left\{ v_N(x_{1}-x_{2}) b^*_{x_1} b^*_{x_2}  b_{x_1}b_{x_2} \right\}\ .
\end{align}
\end{subequations}
where $\tilde h$ is the modified Hartree operator \eqref{pphieq}.
Let us remark that, in complete analogy to (67c) in \cite{GM}, the quadratic terms \eqref{quad1}+\eqref{quad2}+ \eqref{quad3} can be written concisely
and explicitly as
\begin{align*}
\eqref{quad1}+\eqref{quad2}+ \eqref{quad3}=\H_{\tilde G}-
\I\left(\begin{matrix}
\tilde w^{T}&\overline{\tilde f}
\\
-\tilde f&-\tilde w
\end{matrix}
\right)
\end{align*}
where
\begin{align*}
&\tilde f:=\big(\tilde \S(\sh)+\chb\circ(v_N \Lambda) \big)\circ\ch  -\big(\tilde \W(\chb)-\sh\circ \overline{(v_N \Lambda)}\big)\circ \sh
\\
&\tilde w:=\big(\tilde \W(\chb)-\sh\circ\overline{(v_N\Lambda)}\big)\circ \chb -\big(\tilde \S(\sh)+\chb\circ ( v_N \Lambda)\big)\circ\shb\\
&\H_{\tilde g} = \int \tilde g(x, y)a^*_x a_y dx dy\\
\end{align*}
Also, $\I$ is the Lie algebra isomorphism used  in our previous papers \cite{GMM1}-\cite{GMM}
(see for instance formula (27) in \cite{GM}) is
\begin{align*}
\I\left(\begin{matrix}
\tilde w^{T}&\overline{\tilde f}
\\
-\tilde f&-\tilde w
\end{matrix}
\right)
=-\frac{1}{2}\int dxdy\left\{
\tilde w(y,x)a_xa^*_y+\tilde w(x,y)a_x^*a_y - \tilde f(x,y)a_x^* a_y^*- \overline { \tilde f} (x,y)a_x a_y\right\}
\end{align*}
In this notation,  $\X_2$ is a multiple of $\tilde f $, thus $\tilde f=0$ if our equations are satisfied.

If we also put the quadratics in normal order,
the above formula becomes
\begin{align*}
 &\H_{red}
=X_0(t)+  \\
&  N^{1/2}
\int dx\left\{
\tilde h(\phi(t,x))b^{\ast}_{x}
+\bar{\tilde h}(\phi(t,x))b_x \right\}  \\
&+
\H_{\tilde g} + \Nor \Bigg(\I\left(\begin{matrix}
\tilde w^{T}&\overline{\tilde f}
\\
-\tilde f&-\tilde w
\end{matrix}
\right)
\\
 &- \frac{1}{\sqrt N}
\int dx_{1}dx_{2} \left\{
 v_N(x_{1}-x_{2}) \left(\overline{\phi}(x_{2}) b^{\ast}_{x_{1}} b_{x_{1}}b_{x_{2}} + \phi(x_{2})b_{x_{1}}^*b_{x_{2}}^*b_{x_{1}}\right)
\right\}\\
&-\frac{1}{2 N}
\int dx_{1}dx_{2} \left\{ v_N(x_{1}-x_{2}) b^*_{x_1} b^*_{x_2}  b_{x_1}b_{x_2} \right\}\Bigg) \ .
\end{align*}
where $X_0$ is written down explicitly in Section 6 on \cite{GMM}.
In conclusion, if $\phi$ and $k$ satisfy our equation $X_1=0, X_2=0$, then $\tilde h(\phi(t,x))=0$,  $\tilde f =0$ and
\begin{align}
\P&: \label{pform}
= H_{red}-\H -X_0\\
&=
\H_{\tilde g} -\H_{1}
+ \Nor \bigg(\I\left(\begin{matrix}
\tilde w^{T}&0
\\
0&-\tilde w \notag
\end{matrix}
\right) \bigg)\\
 +& \Nor \bigg(- \frac{1}{2 \sqrt N} \notag
\int dx_{1}dx_{2} \left\{
 v_N(x_{1}-x_{2}) \left(\overline{\phi}(x_{2}) b^{\ast}_{x_{1}} b_{x_{1}}b_{x_{2}} + \phi(x_{2})b_{x_{1}}^*b_{x_{2}}^*b_{x_{1}}\right)
\right\}\\
&-\frac{1}{2 N} \notag
\int dx_{1}dx_{2} \left\{ v_N(x_{1}-x_{2}) b^*_{x_1} b^*_{x_2}  b_{x_1}b_{x_2} \right\}\bigg)\\
&+\frac{1}{2 N}
\int dx_{1}dx_{2} \left\{ v_N(x_{1}-x_{2}) a^*_{x_1} a^*_{x_2}  a_{x_1}a_{x_2} \right\} \notag
 \ .
\end{align}
Remark that the third and fourth term are  $-\frac{1}{\sqrt N}\Nor \left(e^{\B} [\A, \V] e^{-\B}\right)$ and $-\frac{1}{ N}\Nor \left(e^{\B} \V e^{-\B}\right)$ (see  Section 5 of \cite{GM}).

\section{Estimates for the error term }\label{fockerror}
In this section we  apply the estimates of Corollary \eqref{shest} in order to estimate the error.

We proceed as in our previous papers \cite{GMM1}-\cite{GM}, using the identity
\begin{align*}
\big\Vert\psi_{exact}(t)-
e^{i \int X_0(t) dt}\psi_{appr}(t)\big\Vert_{\F}
 = \big\Vert e^{-i \int X_0(t) dt}\psi_{red} - \Omega\big\Vert_{\F}
 \end{align*}
and estimate the right hand side term using the equation
\begin{align*}
\left(\frac{1}{i}\frac{\partial}{\partial t} - \H_{red} + X_0\right)\left( e^{-i \int X_0(t) dt}\psi_{red} - \Omega\right)
=(0, 0, 0, X_3, X_4, 0, \cdots):=\tilde X
\end{align*}
Recall $X_0$ is the zeroth order term in $\H_{red}$.\\
Denote $E=e^{-i \int X_0(t) dt}\psi_{red} - \Omega$, so that
\begin{align}
&\left(\frac{1}{i}\frac{\partial}{\partial t} - \H_{red} + X_0\right)E \label{Edef}
=(0, 0, 0, X_3, X_4, 0, \cdots) := \tilde X\\
&E(0, \cdot)=0 \notag
\end{align}
In the proof that follows we will write
\begin{align*}
\S_F&:=\frac{1}{i}\frac{\partial}{\partial t} - \H_{red} + X_0\\
&:=\S_D -\P
\end{align*}
where
\begin{align*}
&\S_D=\frac{1}{i}\frac{\partial}{\partial t} - \H\\
&\H =  \int a^*_x\Delta a_x - \frac{1}{2N}\V
\end{align*}
Thus $\H$ is the original (unconjugated) Fock space Hamiltonian \eqref{FockHamilt1-a}-\eqref{FockHamilt1-c},
 and $\P$ accounts for the rest of the terms:
 \begin{align*}
 \P= H_{red}-\H -X_0
 \end{align*}
 Recall $\H$ acts on  Fock space as
 \begin{align*}
 \Delta -\frac{1}{N}\sum_{j<k}N^{3\beta}v\big(N^{\beta}(x_{j}-x_{k})\big)
 \end{align*}
 interpreted as $0$ on the zeroth slot and  $\Delta$ on the first one.

The terms $X_0$, $X_3$ and $X_4$ were computed in Section 5 of \cite{GM}, see formulas (74c) and (72c).
We only need $X_3$ and $X_4$ here, and they  are (modulo symmetrization and normalization)
\begin{align*}
X_3(y_1, y_2, y_3)=
&\frac{1}{\sqrt N} \int\chhb(y_{1},x_{1})\chh(x_{2},y_{2})v_{N}(x_{1}-x_{2})\phi(x_{2})\shh(y_{3},x_{1})dx_1dx_2\\
&= \frac{1}{\sqrt N}\left(v_{N}(y_{1}-y_{2})\shh(y_{3},y_{1})\phi(y_{2}) + LOT\right)\\
X_4(y_1, y_2, y_3, y_4)=
&\frac{1}{ N}\int\chhb(y_{1},x_{1})\chh(x_{2},y_{2})v_{N}(x_{1}-x_{2})\shh(y_{3},x_{1})\shh(x_{2},y_{4})dx_1dx_2\\
&=\frac{1}{ N}v_{N}(y_{1}-y_{2})\left(\shh(y_{3},y_{1})\shh(y_{2},y_{4}) +LOT\right)
\end{align*}
where the lower order terms $LOT$ come from the $p$ component of $\ch=\delta(x-y)+p$.

 We will use the following Strichartz norms $S$ and dual Strichartz norms $S'$
 for the equation $\S u=f$ where $u=u(t, x_1, \cdots x_n)$, $x_i \in \mathbb R^3$, $t \in [0, T_0]$.
 \begin{definition} Define the following norms for $0 \le t \le T_0$:
 \begin{align*}
 &\|u\|_S=\max\{ \|u\|_{L^{\infty}(dt) L^2(dx_1 \cdots dx_n)},
   \|u\|_{L^{2}(dt) L^6(d(x_1- x_2)) L^2(d(x_1+x_2) \cdots dx_n)},\\
 & \mbox{and all other permutaions}\}
 \end{align*}
 and
 \begin{align*}
 &\|u\|_{S'}=\min\{ \|u\|_{L^{1}(dt) L^2(dx_1 \cdots dx_n)},
  \|u\|_{L^{2}(dt) L^{6/5}(d(x_1- x_2)) L^2(d(x_1+x_2) \cdots dx_n)} ,\\
 & \mbox{and all other permutaions}\}
 \end{align*}
Also, if $X$ is an element of Fock space with finitely many non-zero components  $X_0, \cdots, X_k$, we denote
\begin{align*}
\|X\|_S
= \max\{|X_0|, \|X_1\|_{S}, \cdots, \|X_k\|_{S}\},
\end{align*}
and similarly
\begin{align*}
\|X\|_{S'}
= \max\{|X_0|, \|X_1\|_{S'}, \cdots, \|X_k\|_{S'}\}
\end{align*}

 \end{definition}
  As it is well known, using the $T-T^*$ argument and the Christ-Kiselev lemma, as well as estimates for the homogeneous
  equation such as \eqref{Xuwenest} if $\S u =f$, $u(t=0)=0$, then $\|u\|_{S} \lesssim \|f\|_{S'}$.

  Furthermore, if
  $\beta < 1$,
  \begin{align*}
  & \frac{1}{N}\| v_N(x_1-x_2) u\|_{S'} \le
  \frac{1}{N}\| v_N(x_1-x_2) u\|_{L^{2}(dt) L^{6/5}(d(x_1- x_2)) L^2(d(x_1+x_2) \cdots dx_n)}\\
  &\le \frac{1}{N}\| v_N\|_{L^{3/2}}  \|u\|_{L^{2}(dt) L^6(d(x_1- x_2)) L^2(d(x_1+x_2) \cdots dx_n)},\\
  &\le C N^{- power} \|u\|_{S}
  \end{align*}
  so we can treat the potential as a perturbation
  and conclude the following:
  \begin{lemma} \label{strichw}
  Let $f$ be a Fock vector with zero entries past the $k$th slot ($k=21$ in the application that follows). Assume
  \begin{align*}
 & \S_{D} u =f\\
  &u(t=0)=0
  \end{align*}
  then
  \begin{align*}
  \|u\|_S \lesssim \|f\|_{S'}
  \end{align*}
  and notice that $\sup_t\|u\|_{\F} \lesssim \|u\|_{S}$.
  \end{lemma}

 The main result of this section, which completes the proof of Theorem \eqref{GM3thm}, is
 \begin{theorem} Let $\beta < 2/3$, and let $E$ satisfy \eqref{Edef}.
 Assume $\Lambda$, $\phi$ and $\Gamma$ satisfy the estimates of Theorem \eqref{mainNL} as well as Corollary \eqref{shest}.
 Then
 \begin{align*}
 \|E\|_{\F} \lesssim N^{-\frac{1}{6}}
 \end{align*}
  for $0 \le t \le T_0$.
 \end{theorem}
 \begin{proof}
 Recall the splitting
\begin{align*}
\S_F&:=\frac{1}{i}\frac{\partial}{\partial t} - \H_{red} + X_0\\
&:=\S_D -\P
\end{align*}
where $\S_D$ is a diagonal Schr\"odinger operator.
 Also, $\P$ was computed at the end of Section \eqref{reducedham}.

 We observe $\|X_4\|_{L^2}= c N^{\frac{3 \beta}{2}-1} <<1$, but $\|X_3\|_{L^2}= c N^{\frac{3 \beta -1}{2}} >>1$
 so we cannot use energy estimates and are forced to use the Strichartz estimates of Lemma \eqref{strichw}.
 We proceed to solve
\begin{align}
& (S_D - \P) E = \X =(0, 0, 0, X_3, X_4, \cdots) \label{fockspaceest}\\
 & E(0)=0 \notag
\end{align}
by iteration. We will iterate 4 times, and by the end the vector on the right hand side will have at most
$5 + 4 \times 4$ non-zero entries. ($\P$ is a fourth order operator in $a$ and $a^*$).

  Since
 $\|\frac{1}{\sqrt N} v_N \|_{L^{6/5}}=c N^{\frac{\beta -1}{2}} \le c N^{-\frac{1}{6}}$,
 we have $\|X\|_{S'} \lesssim N^{-\frac{1}{6}}$ and therefore the first iterate defined by
 \begin{align*}
& S_D  E_1 = \X =(0, 0, 0, X_3, X_4, \cdots) \\
 & E_1(0)=0
\end{align*}
satisfies
$\|E_1\|_{S} \lesssim N^{-\frac{1}{6}}$. If $\P$ were bounded on the first few slots of Fock space we would be done. Unfortunately, this is not the case,  the norm of $\P$ on the first few slots of Fock space is $\le C N^{\frac{3 \beta -1}{2}}\le C N^{\frac{1}{2}}$.

In Section \eqref{21} below we prove that $\P=\P_1+\P_2$ where $\P_2$ is bounded on (the first 21 slots of) Fock space, while
$\|\P_1 E \|_{\F} \lesssim N^{\frac{1}{2}}\|E \|_{\F}$ but also
$\|\P_1 E \|_{S} \lesssim N^{-\frac{1}{6}}\|E \|_{S'}$. The reasons behind these bounds are
$\|\frac{1}{\sqrt N} v_N \|_{L^{6/5}}=c N^{\frac{\beta -1}{2}} \le c N^{-\frac{1}{6}}$,
$\|\frac{1}{\sqrt N} v_N \|_{L^{2}}=c N^{\frac{3\beta }{2}-1} \le c N^{\frac{1}{2}}$.
Consider the iteration
 \begin{align*}
&\left( S_D - \P \right)\left(  E_1 + \S_D^{-1} \P_1  E_1 + \left(\S_D^{-1}\P_1 \right)^2E_1+ \left(\S_D^{-1}\P_1 \right)^3 E_1\right)\\
& = \X
-\P_2\left(  E_1 + \S_D^{-1} \P_1  E_1 +  \left(\S_D^{-1}\P_1 \right)^2E_1\right) - \P \left(\S_D^{-1}\P_1 \right)^3 E_1
 \end{align*}
Assuming the estimates of Propositions \eqref{p1terms}, \eqref{p2terms}, \eqref{p22terms},
we have the following fixed time estimates for $0 \le t \le T_0$:
 \begin{align*}
 &\|\P_2\left(  E_1 + \S_D^{-1} \P_1  E_1 +  \left(\S_D^{-1}\P_1 \right)^2E_1\right)\|_{\F}\\
 &\lesssim \|\left(  E_1 + \S_D^{-1} \P_1  E_1 +  \left(\S_D^{-1}\P_1 \right)^2E_1\right)\|_{\F}\\
 &\lesssim N^{-\frac{1}{6}}
 \end{align*}
 and
 \begin{align*}
 \|\P \left(\S_D^{-1}\P_1 \right)^3 E_1\|_{\F}
 \lesssim N^{\frac{1}{2}}  N^{-\frac{4}{6}} \lesssim N^{-\frac{1}{6}}
 \end{align*}
 Now we can use energy estimates to compare the above fourth iterate with the exact solution to \eqref{Edef}
 as in our previous papers \cite{GMM1}-\cite{GM}, and conclude
 the proof of the theorem.
 \end{proof}

All we have to do now is estimate the norm $\P$, restricted to the first 21 slots of Fock space.
In order to do that, we need some precise information on the terms in $\P$, which was developed in Section
\eqref{reducedham}.

\section{Estimates for the norm of $\P$ restricted to the first 21 slots of Fock space} \label{21}

In this section we will use repeatedly the fact that, if $\beta \le 2/3$,   $\frac{1}{N}\| v_N\|_{L^{2}} \lesssim 1$,
 $\frac{1}{\sqrt N}\| v_N\|_{L^{2}} \lesssim N^{ \frac{1}{2} }$,
 $\frac{1}{\sqrt N}\| v_N\|_{L^{\frac{6}{5}}} \lesssim N^{- \frac{1}{6} }$.
Recall formula \eqref{pform} defining $\P$..

We start by defining the term $\P_1$ and proving its properties.
\begin{proposition} \label{p1terms}
Let $\P_1$ is a linear combination of the following terms coming from
\begin{align*}
\frac{1}{\sqrt N}\Nor \left(e^{\B}[\A, \V] e^{-\B}\right)
\end{align*}
\begin{align*}
&T_1=\frac{1}{\sqrt N}
\int dx_{1}dx_{2}
 v_N(x_{1}-x_{2}) \overline{\phi}(x_{2})a(\shb(x_1, \cdot)) a(\chb(x_1, \cdot))a(\chb(x_2, \cdot))\\
&T_2=\frac{1}{\sqrt N}
\int dx_{1}dx_{2}
 v_N(x_{1}-x_{2}) \phi(x_{2})a^*(\sh(x_1, \cdot)) a^*(\ch(x_1, \cdot))a^*(\ch(x_2, \cdot))\\
&T_3=\frac{1}{\sqrt N}
\int dx_{1}dx_{2}
 v_N(x_{1}-x_{2}) \overline{\phi}(x_{2}) a^{\ast}(\ch(x_1, \cdot)) a(\chb(x_1, \cdot))a(\chb(x_2, \cdot)) \\
 &T_4=\frac{1}{\sqrt N}
\int dx_{1}dx_{2}
 v_N(x_{1}-x_{2})  \phi(x_{2})a^*(\ch(x_1, \cdot))a^*(\ch(x_2, \cdot))a((\chb(x_1, \cdot))
\end{align*}

If $X$ is a Fock space vector which has non-zero entries only in the first $k$ ($k=21$) slots and $T$ is one of the above $T_i$,
then
\begin{align}
&\|T X \|_{\F} \lesssim N^{\frac{1}{2}}\|X\|_{\F} \label{ffockest1}\\
&\|T X \|_{S'} \lesssim N^{-\frac{1}{6}}\|X\|_{S} \label{ssest1}
\end{align}

 \end{proposition}
 \begin{remark} It is easy to see the meaning of these terms.
   $T_1^*=T_2$,   $T_2 \Omega =X_3$ while $T_3$ and $T_4$ with $\ch$ replaced by $\delta(x-y)$ correspond
 to the unconjugated  $\frac{1}{\sqrt N} [\A, \V]$.
 \end{remark}
\begin{proof}
In treating the above terms recall $\ch(t, x, y)=\delta(x-y) + p(t, x, y)$. The worst terms are always obtained from
the $\delta$ term (because composition with $p$ is bounded on $L^2$), so we will only discuss these. Also, recall $\phi$ is known to be bounded, and $L^2$. Replacing  $ \ch(t, x, y)$ by $\delta(x-y)$, $T_1$ gets replaced by
\begin{align*}
\frac{1}{\sqrt N}
\int dx_{1}dx_{2}
 v_N(x_{1}-x_{2}) \overline{\phi}(x_{2})a(\shb(x_1, \cdot)) a_{x_1}a_{x_2}
 \end{align*}
This acts on a Fock space vector of the form $(0, \cdots, F(x_1, \cdots x_n, 0, \cdots)$
 as
 \begin{align*}
 \int
 \frac{1}{\sqrt N}v_N(x_{1}-x_{2}) \overline{\phi}(x_{2})\shb(x_1, z))F(x_1, x_2, z, \cdots) dx_1 dx_2 dz
 \end{align*}
 Now we use
 \begin{align*}
 &\|\frac{1}{\sqrt N}v_N(x_{1}-x_{2}) \overline{\phi}(x_{2})\shb(x_1, z))\|_{L^2(dx_1 dx_2 dz)}\\
& \le \sup_{x_1}\|\shb(x_1, z)\|_{L^2(dz)}\| \|\frac{1}{\sqrt N}v_N\|_{L^2}\|\phi\|_{L^2} \lesssim N^{\frac {1}{2}}
 \end{align*}
which implies \eqref{ffockest1}, and also
  \begin{align*}
 &\|\frac{1}{\sqrt N}v_N(x_{1}-x_{2}) \overline{\phi}(x_{2})\shb(x_1, z))\|_{L^{6/5}(d(x_1-x_2))L^2((dx_1 +x_2) dz)}\\
& \le \sup_{x_1}\|\shb(x_1, z)\|_{L^2(dz)}\| \|\frac{1}{\sqrt N}v_N\|_{L^{6/5}}\|\phi\|_{L^2} \lesssim N^{-\frac {1}{6}}
 \end{align*}
 which implies the fixed time estimate
 \begin{align*}
 \|T_1(F)(t)\|_{L^2} \lesssim N^{-\frac {1}{6}} \|F\|_{L^6(d(x_1-x_2))L^2(d(x_1+x_2) dx_3 \cdots)}
 \end{align*}
 Now take $L^1$ in time, and dominate that by $L^2$ in time on the right hand side, since $t \in [0, 1]$.
 This proves
 \begin{align*}
 \|T_1(F)(t)\|_{S'} \lesssim N^{-\frac {1}{6}} \|F\|_{S}
 \end{align*}
 The estimate for $T_2$ is the dual of this argument.
The term $T_3$ is (after replacing $\ch$ by $\delta$ )
\begin{align*}
& \frac{1}{\sqrt N}
\int dx_{1}dx_{2}
 v_N(x_{1}-x_{2}) \overline{\phi}(x_{2}) a^{\ast}_{x_{1}} a_{x_{1}}a_{x_{2}}
 \end{align*}
 This acts on $F$ by
 \begin{align*}
 & \frac{1}{\sqrt N}
\int
 v_N(x_{1}-x_{2}) \overline{\phi}(x_{2}) F(x_1, x_2, \cdots) dx_2
 \end{align*}
 The variables $x_j$, $j \ge 3$ are passive, so, without loss of generality,  we  take $F=F(x_1, x_2)$
 The $L^2$ bound is immediate, and for the $S$, $S'$ bound write $F(x_1, x_2)=G(x_1-x_2, x_1+x_2)$ and
 \begin{align*}
 & \frac{1}{\sqrt N}\|
\int
 v_N(x_{1}-x_{2}) \overline{\phi}(x_{2}) G(x_1-x_2, x_1+ x_2) dx_2\|_{L^2(dx_1)}\\
 &= \frac{1}{\sqrt N}\|
\int
 v_N(x_{2}) \overline{\phi}(x_1-x_{2}) G(x_2, 2 x_1- x_2) dx_2\|_{L^2(dx_1)}\\
 &\le \frac{1}{\sqrt N} \|\phi\|_{L^{\infty}} \|v_N\|_{L^{6/5}}\|G\|_{L^6 L^2}
 = c N^{-\frac{1}{6}} \|\phi\|_{L^{\infty}} \|F\|_{L^6(d(x-y)) L^2(d(x+y))}
 \end{align*}
 The bounds for $T_4$ are easy, and left to the reader.
\end{proof}

\begin{proposition} \label{p2terms}
The terms of $\P$ other than $T_1, \cdots T_4$  coming from
\begin{align*}
\frac{1}{\sqrt N}\Nor \left(e^{\B}[\A, \V] e^{-\B}\right)
\end{align*}
have bounded operator norm on the first 21 slots of Fock space.
\end{proposition}
\begin{proof}
To prove the proposition,
we have to estimate the terms in
 \begin{align*}
\frac{1}{\sqrt N} \Nor \left(
\int dx_{1}dx_{2}\left\{
 v_N(x_{1}-x_{2}) \left(\overline{\phi}(x_{2}) b^{\ast}_{x_{1}} b_{x_{1}}b_{x_{2}} + \phi(x_{2})b_{x_{1}}^*b_{x_{2}}^*b_{x_{1}}\right)
\right\}\right)
\end{align*}

The two terms are dual to each other, so will just estimate the first one.
In principle, there are $2^3$ terms to estimate, following the pattern
\begin{subequations}
\begin{align}
a_1a_1a_2 \label{t1}\\
a^*_1a_1a_2  \label{t2}\\
a_1 a_1^*a_2 \label{t3}\\
\cdots  \notag \\
a_1^*a_1a_2^* \label{t6}\\
\cdots
\end{align}
\end{subequations}
The two terms, \eqref{t1} and \eqref{t2} are $T_1$ and $T_3$ ,  which have been discussed in the previous lemma.
During the proof, we will comment on where some of the contraction terms go, but point out we do not need to estimate them.

The term \eqref{t3} stands for
\begin{align*}
\frac{1}{\sqrt N}
\int dx_{1}dx_{2}
 v_N(x_{1}-x_{2}) \overline{\phi}(x_{2})a(\shb(x_1, \cdot)) a^*(\sh(x_1, \cdot))a(\chb(x_2, \cdot))
 \end{align*}
Here the contraction $[a(\shb(x_1, \cdot)), a^*(\sh(x_1, \cdot))]= (\shb \circ \, \sh )(x_1, x_1)$
pairs up with $|\phi|^2$ to form $Tr \Gamma$ in the formula for $\tilde h$, and we do not have to estimate it.
The ordered term
\begin{align*}
\frac{1}{\sqrt N}
\int dx_{1}dx_{2}
 v_N(x_{1}-x_{2}) \overline{\phi}(x_{2}) a^*(\sh(x_1, \cdot) a(\shb(x_1, \cdot)))a(\chb(x_2, \cdot))
 \end{align*}
 acts on $F$ as
 \begin{align*}
\frac{1}{\sqrt N}
\int
 v_N(x_{1}-x_{2}) \overline{\phi}(x_{2}) \sh(x_1, x_3) \shb(x_1, z))F(x_2, z, \cdots) dx_1 dx_2 dz
 \end{align*}
 and has operator norm $\lesssim \frac{1}{\sqrt N}$.

 The term \eqref{t6} stands for
 \begin{align*}
\frac{1}{\sqrt N}
\int dx_{1}dx_{2}\overline{\phi}(x_{2}) a^*(\ch(x_1, \cdot))a(\chb(x_1, \cdot)) a^*(\sh(x_2, \cdot))
 \end{align*}
 Here the estimate is not true for the term involving the contraction $[a(\chb(x_1, \cdot)),  a^*(\sh(x_2, \cdot))] = \frac{1}{2} \sht(x_1, x_2)$. This term gets paired up with $\phi(x_1) \phi(x_2)$ to form $\Lambda$, and becomes part of $\tilde h ( \phi)$.
 The remaining ordered term has norm $\lesssim \frac{1}{\sqrt N}$.

All remaining terms have operator norms $\lesssim \frac{1}{\sqrt N}$, as can be easily checked.

\end{proof}

\begin{proposition} \label{p22terms}
The terms $\P$ coming from
\begin{align*}
\frac{1}{ 2N}\Nor \bigg(e^{\B} \V e^{-\B}\bigg) -\frac{1}{2 N} \V
\end{align*}
 have bounded operator
norm  on the first 21 slots of Fock space.
\end{proposition}

\begin{proof}

We have to estimate the $2^4-1$ terms in
\begin{align*}
\frac{1}{2N}\Nor \left(\int v_N(x_1-x_2) b_{x_1}^* b_{x_2}^* b_{x_1} b_{x_2} dx_1 dx_2-
\frac{1}{2N}\int v_N(x_1-x_2) a_{x_1}^* a_{x_2}^* a_{x_1} a_{x_2} dx_1 dx_2\right),
\end{align*}
During the proof, we will comment on where some of the contraction terms go, but point out we do not need to estimate them.
Recall the formula \eqref{conjf}.

The operators $b_{x}-a_{x}$ and  $b^*_{x}-a^*_{x}$ are  linear combinations of $a(f(x, \cdot)$,  $a^*(f(x, \cdot)$
where  $f$ is one of $\sh$, $p$, or their complex conjugates , and satisfies the estimates
of Lemma \eqref{shest}.

We look at  all possible terms in $ b_{x_1}^* b_{x_2}^* b_{x_1} b_{x_2}$. Schematically,
\begin{subequations}
\begin{align}
a_1a_2a_1a_2\label{start}\\
a^*_1a_2a_1a_2 \label{s2}\\
a_1 a_2^*a_1a_2 \label{s3}\\
a_1^*a_2^*a_1a_2 \label{s4}\\
\cdots \notag\\
a_1^*a_2^*a_1^*a_2 \label{s8}\\
a_1^*a_2^*a_1a_2^* \label{s12}\\
\cdots \notag\\
a_1^*a_2^*a_1^*a_2^* \label{end}
\end{align}
\end{subequations}
and estimate some typical ones.

 The term \eqref{start} means
 \begin{align*}
\frac{1}{2N}\int v_N(x_1-x_2) a(\shb(x_1, \cdot)  a(\shb(x_2, \cdot) a(\chb(x_1, \cdot) a(\chb(x_2, \cdot) dx_1 dx_2
\end{align*}
and this is dual to \eqref{end}
 \begin{align*}
\frac{1}{2N}\int v_N(x_1-x_2) a^*(\ch(x_1, \cdot)  a^*(\ch(x_2, \cdot) a^*(\sh(x_1, \cdot) a^*(\sh(x_2, \cdot) dx_1 dx_2
\end{align*}

The first one acts by integration against, and the second one acts as a (normalized, symmetrized) tensor product with
\begin{align*}
X_4(y_1, y_2, y_3, y_4)=
\frac{1}{ N}\int\chhb(y_{1},x_{1})\chh(x_{2},y_{2})v_{N}(x_{1}-x_{2})\shh(y_{3},x_{1})\shh(x_{2},y_{4})dx_1dx_2
\end{align*}
Treating $\ch(x_1, x_2)=\delta(x_1-x_2) + \p$ we will only include $\delta$ in our calculation since this is always the worst case. With this simplification, the norm of the above operators is dominated by
\begin{align*}
\|\frac{1}{2 N}v_{N}(y_{1}-y_{2})\sh(y_{3},y_{1})\sh(y_{2},y_{4})\|_{L^2(dy_1dy_2dy_3dy_4)}\\
\le
\|\frac{1}{2 N}v_{N}\|_{L^2}\sup_x\|\sh(x, y)\|_{L^2(dy)}\|\sh(x, y)\|_{L^2(dx dy)} \lesssim  1
\end{align*}

Next we consider \eqref{s2} and the dual \eqref{s8}. It suffices to treat just one, say \eqref{s8}.
This term stands for
 \begin{align*}
\frac{1}{2N}\int v_N(x_1-x_2) a^*(\ch(x_1, \cdot)  a^*(\ch(x_2, \cdot) a^*(\sh(x_1, \cdot) a(\chb(x_2, \cdot) dx_1 dx_2
\end{align*}

For simplicity, replacing $\ch(x_1, x_2)$ by $\delta(x_1-x_2) $,  the above term acts on $F(\cdots)$
(actually, the vector$(0,  \cdots, F, 0, \cdots)$
producing
\begin{align*}
G(x_1, x_2, x_3, \cdots)=c\frac{1}{N} v_N(x_1-x_2)\sh(x_1, x_3) F(x_2, \cdots)
\end{align*}
which is easily seen to have $L^2$ norm $\lesssim   \|F\|_{L^2}$.
(first do $L^2(dx_3)$, then $L^2(dx_1)$, leaving $L^2(dx_2 d(\cdots))$ last).
Next we consider \eqref{s3} and the dual \eqref{s12}, written explicitly as
\begin{align*}
 &\frac{1}{2N}\int v_N(x_1-x_2) a(\shb(x_1, \cdot))  a^*(\ch(x_2, \cdot)) a(\chb(x_1, \cdot)) a(\chb(x_2, \cdot)) dx_1 dx_2\\
 &\frac{1}{2N}\int v_N(x_1-x_2) a^*(\ch(x_1, \cdot)  a^*(\ch(x_2, \cdot)) a(\chb(x_1, \cdot)) a^*(\sh(x_2, \cdot)) dx_1 dx_2
\end{align*}
The estimate would not be true for these terms as they stand, but becomes true after putting them in normal order.
For the first one, the contraction
\begin{align*}
 \frac{1}{2N}\big[a(\shb(x_1, \cdot)) , a^*(\ch(x_2, \cdot))\big]=\frac{1}{4N}\shbt(x_1, x_2)
 \end{align*}
gets paired with $\overline\phi(x_1)\phi(x_2)$ from the quadratic term \eqref{badquad} to become $\overline \Lambda(x_1, x_2)$.
We are left with estimating
\begin{align*}
 &\frac{1}{2N}\int v_N(x_1-x_2) a^*(\ch(x_2, \cdot)) a(\shb(x_1, \cdot))  a(\chb(x_1, \cdot)) a(\chb(x_2, \cdot)) dx_1 dx_2
\end{align*}
With the usual simplification $\ch=\delta + \cdots$ this acts on $F$ as
\begin{align*}
 &\frac{1}{2N}\int v_N(x_1-x_2) \shb(x_1, z)) F(x_1, x_2, z \cdots) dx_1 dz
\end{align*}
which is easily seen to have norm $\lesssim   \|F\|_{L^2}$.

For \eqref{s4} we estimate the difference
\begin{align*}
&\frac{1}{2N}\int v_N(x_1-x_2) a^*(\ch(x_1, \cdot)) a^*(\ch(x_2, \cdot))a(\chb(x_1, \cdot)) a(\chb(x_2, \cdot)) dx_1 dx_2\\-
&\frac{1}{2N}\int v_N(x_1-x_2) a_{x_1}^* a_{x_2}^* a_{x_1} a_{x_2} dx_1 dx_2,
\end{align*}
For instance,
\begin{align*}
&\frac{1}{2N}\int v_N(x_1-x_2) a^*(p(x_1, \cdot)) a^*(\ch(x_2, \cdot))a(\chb(x_1, \cdot)) a(\chb(x_2, \cdot)) dx_1 dx_2
\end{align*}
replacing the $\ch$ with $\delta$, acts by
\begin{align*}
&\frac{1}{2N}\int v_N(x_1-x_2) p(x_1, z)F(x_1, x_2) dx_1
\end{align*}
which has $L^2(dx_2 dz)$  norm $\lesssim \|\frac{1}{2N} v_N\|_{L^2} \|p\|_{L^{\infty}L^2}\|F\|_{L^2}
\lesssim \|F\|_{L^2}
$.

All other terms are similar.

\end{proof}

 \begin{lemma} Terms of $\P$ coming from
 \begin{align*}
 \H_{\tilde g} -\H_{1}
+ \Nor \bigg(\I\left(\begin{matrix}
\tilde w^{T}&0
\\
0&-\tilde w
\end{matrix}
\right)\bigg)
\end{align*}
are bounded from the first 21 slots of Fock space to Fock space uniformly in $N$.
 \end{lemma}

 \begin{proof}
Since $\Gamma $ is known to be bounded uniformly in $N$, this is clear for $\H_{\tilde g} -\H_{1}$ which equals\\
$
\int \left(v_N * (Tr \Gamma)(t, x) \delta(x-y)
 +v_N(x-y)\Gamma(t, x, y)\right)a_x^*a_y dx dy
$.
Also, recall
\begin{align*}
&\tilde w:=\big(\tilde \W(\chb)-\sh\circ\overline{(v_N\Lambda)}\big)\circ \chb -\big(\tilde \S(\sh)+\chb\circ ( v_N \Lambda)\big)\circ\shb
\end{align*}
The equation $\tilde f =0$ implies the identity (see section 5 of \cite{GM} for a similar calculation)
we get
\begin{align}
\tilde w(y,x)&=\tilde \W\big(\chb\big)\circ \big(\chb\big)^{-1}-\sh\circ \overline{(v_N \Lambda)}\circ\big(\chb\big)^{-1} \notag\\
&=-\frac{1}{2}\Big(\big(\chb\big)^{-1}\circ(v_N \Lambda)\circ \shb -\sh\circ\overline{(v_N \label{traceterm} \Lambda)}\circ\big(\chb\big)^{-1}\Big)\\
&-\frac{1}{2}\Big[\tilde \W(\chb),\big(\chb\big)^{-1}\Big]\ .\notag
\end{align}
The
equation \eqref{wweq}  together with \eqref{l2} show that
$\tilde \W\big(\chbt\big)$ is in $L^2$ uniformly in $N$. Finally,
the spectral theorem (as in (30) of \cite{GMM2}) is used to express  $\tilde \W \left(\chb \right)$ in terms of
$\tilde \W \left(\chbt\right)$  we see that $\tilde w$ is a Hilbert-Schmidt operator uniformly in $N$.
Its trace
\eqref{traceterm} is part of $X_0$, and
\begin{align*}
 \Nor \bigg(\I\left(\begin{matrix}
\tilde w^{T}&0
\\
0&-\tilde w
\end{matrix}
\right)\bigg)
\end{align*}
is bounded on the first five slots of Fock space uniformly in $N$.
\end{proof}

\end{document}